\documentclass{amsart}
\usepackage{amssymb}
\usepackage{stmaryrd} 
\usepackage{amsmath} 
\usepackage{amscd}
\usepackage{amsthm}
\usepackage{amsbsy}
\usepackage{tikz}
\usetikzlibrary{arrows,positioning}
\usepackage[margin=1in]{geometry}
\usepackage{commath, longtable}
\usepackage{comment, enumerate}
\usepackage{geometry}
\usepackage[matrix,arrow]{xy}
\usepackage{hyperref}
\usepackage{mathrsfs}
\usepackage{color}
\usepackage{float}
\usepackage{mathtools,caption}
\usepackage[table,dvipsnames]{xcolor}
\usepackage{tikz-cd}
\usetikzlibrary{shapes.geometric,calc,angles}
\usetikzlibrary{decorations.markings, arrows.meta}
\tikzset{
mid arrow/.style={postaction={decorate, decoration={
markings,
mark=at position .5 with {\arrow{Straight Barb}}
}}},
}
\tikzset{
dot/.style = {circle, fill, minimum size=#1,
              inner sep=0pt, outer sep=0pt},
dot/.default = 2pt 
}
\usepackage{asymptote}

\usepackage{listings}
\usepackage{longtable}
\usepackage[utf8]{inputenc}
\usepackage[OT2,T1]{fontenc}
\usepackage[normalem]{ulem}
\usepackage{hyperref}
\usepackage[customcolors]{hf-tikz}
\usepackage{enumitem}
\newlist{primenumerate}{enumerate}{1}
\setlist[primenumerate,1]{label={\arabic*$'$}}
\hypersetup{
 colorlinks=true,
 linkcolor=DarkOrchid,
 filecolor=blue,
 citecolor=olive,
 urlcolor=orange,
 pdftitle={WIN7 project},
 }
\usepackage{booktabs}

\DeclareSymbolFont{cyrletters}{OT2}{wncyr}{m}{n}
\DeclareMathSymbol{\Sha}{\mathalpha}{cyrletters}{"58}

\newcommand{\CS}[1]{\textcolor{BlueGreen}{\textsf{CS: #1}}}

\newtheorem{theorem}{Theorem}[section]
\newtheorem{lemma}[theorem]{Lemma}

\newtheorem{proposition}[theorem]{Proposition}

\newtheorem{definition}[theorem]{Definition}

\newtheorem*{theorem*}{Theorem}
\newtheorem*{conjecture*}{Defect Conjecture}
\numberwithin{equation}{section}

\theoremstyle{remark}
\newtheorem{remark}[theorem]{Remark}
\newtheorem{example}[theorem]{Example}
\newtheorem{nexample}[theorem]{Non-example}

\newcommand{\GL}{\operatorname{GL}}
\newcommand{\BF}{\mathcal{BF}}

\newcommand{\Div}{\operatorname{Div}}

\newcommand{\Gal}{\operatorname{Gal}}
\newcommand{\val}{\operatorname{val}}

\newcommand{\ord}{\operatorname{ord}}
\newcommand{\cyc}{\operatorname{cyc}}
\newcommand{\Char}{\operatorname{char}}

\newcommand{\Id}{\operatorname{Id}}
\newcommand{\rk}{\operatorname{rank}}
\newcommand{\coker}{\operatorname{coker}}

\newcommand{\In}{\operatorname{In}}
\newcommand{\Out}{\operatorname{Out}}
\newcommand{\Deck}{\operatorname{Deck}}

\newcommand{\xdownarrow}[1]{%
  {\left\downarrow\vbox to #1{}\right.\kern-\nulldelimiterspace}
}
\newcommand{\lda}{\xdownarrow{0.5cm}\vspace{0.05cm}}

\newcommand{\C}{\mathbb{C}}
\newcommand{\Z}{\mathbb{Z}}
\newcommand{\Q}{\mathbb{Q}}

\makeatletter
\theoremstyle{plain} 
\newtheorem*{intr@thm}{\intr@thmname}

\newtheorem*{c@njecture}{\conjn@name}
\newcommand{\myl@bel}[2]{
 \protected@write \@auxout {}{\string \newlabel {#1}{{#2}{\thepage}{#2}{#1}{}} }
 \hypertarget{#1}{}
 } 

\makeatother

\makeatletter
\newcommand{\mylabel}[2]{#2\def\@currentlabel{#2}\label{#1}}
\makeatother

\title[]{Iwasawa Theory of (Directed) Cayley Graphs}

\author[D.~Kundu]{Debanjana Kundu}
\address[DK]{Department of Mathematics and Statistics\\ University of Regina \\ 3737 Wascana Pkwy\\ Regina, SK\\ Canada  S4S 0A2}
\email{debanjana.kundu@uregina.ca}

\author[K.~M\"uller]{Katharina M\"uller}
\address[KM]{Institut für Theoretische Informatik, Mathematik und Operations Research, Universität der Bundeswehr München, Werner-Heisenberg-Weg 39, 85577 Neubiberg, Germany}
\email{katharina.mueller@unibw.de}

\author[A.~Newton]{Alexis Newton}
\address[AN]{Department of Mathematical Sciences\\ High Point University\\ 1 University Pkwy\\ High Point, NC 27262\\ United States of America}
\email{anewton@highpoint.edu}

\author[C.~Stewart]{Chloe Stewart}
\address[CS]{Department of Mathematics \\ Colorado State University\\ 1874 Campus Delivery\\
Fort Collins, CO 80523-1874\\ United States of America}
\email{chloe.stewart@colostate.edu}

\author[Y.~Wang]{Yidi Wang}
\address[YW]{Department of Mathematics, University of Haifa, 199 Abba Khoushy Avenue, Haifa, Israel}
\email{yidiwang.math@gmail.com}


\begin{document}

\begin{abstract}
In this article, we prove the Defect Conjecture for Cayley graphs of abelian groups, dihedral groups, and groups of the form $\Z/p\Z \rtimes \Z/(p-1)\Z$ where $p\geq 3$ is a prime.
We also compute the Iwasawa invariants of the Bowen--Franks groups associated with these graphs in several cases.
\end{abstract}

\maketitle

\section{Introduction}

Iwasawa theory is the study of objects of arithmetic interest over infinite towers of number fields.
It began as a Galois module theory of ideal class groups, initiated by K.~Iwasawa \cite{Iwa59}.
It was extended to the study of abelian varieties in the early 1970's by B.~Mazur \cite{Maz72} and in 1990's to the theory of motives by R.~Greenberg \cite{Gre89, Gre91}.
Within the last decade, the theory has been extended to  (unramified and ramified) coverings of finite connected graphs; see for example \cite{gonet2022iwasawa, Val21, dubose-vallieres, MV23, kataoka2023fitting, MV24, DLRV_Nagoya, Kleine-Mueller4, GV24, kleine2023non, KM26}.
Most recently, A.~Lei and the second-named author initiated a systematic study of Iwasawa theory of directed graphs in \cite{lei2026iwasawa} by studying the Bowen--Franks group of the graph.
In this paper, we further investigate the Iwasawa theory of special classes of directed graphs called Cayley graphs and their associated Bowen--Franks groups.
In particular, a main focus of this article is to provide evidence towards the following conjecture of Lei--M{\"u}ller \cite[Conjecture~7.11]{lei2026iwasawa}.
For precise definitions of terms used in the following conjecture, we refer the reader to Sections~\ref{Iwasawa theory} and \ref{Bowen Franks}.

\begin{conjecture*}
Let $(X_m)_m$ be a constant $\Z_p$-tower of strongly connected digraphs.
The defect $\delta$ is constant throughout the $\Z_p$-tower.    
\end{conjecture*}

\subsection{Main Results}
In this paper, we focus our attention on studying the following classes of Cayley graphs
\begin{enumerate}[label = \textup{(}\roman*\textup{)}]
    \item Cayley graphs of finite abelian groups.
    \item Cayley graphs of semi-direct products (e.g., dihedral groups and $\Z/p\Z \rtimes \Z/(p-1) \Z$).
\end{enumerate}
We prove the Defect Conjecture for the above classes of Cayley graphs; see Theorems~\ref{thm: defect conj for dihedral} and \ref{thm: defect conj for abelian}.
In fact, we are able to show that for our classes of examples the defect is in fact 0, throughout the tower.

In addition, we make modest progress towards other questions of interest in Iwasawa theory.
For Cayley graphs of dihedral groups, we study the size and structure of the Bowen--Franks group of the base graph (when finite); see Section~\ref{sec: size and structure}.
In this setting, we also calculate the Iwasawa invariants at all the primes of the Bowen--Franks groups; see Propositions~\ref{prop 3.6} and \ref{prop 3.12}.
The calculations for the Cayley graphs of dihedral groups can be extended in a straightforward manner to groups of the form $\Z/p\Z \rtimes \Z/(p-1)\Z$; see Section~\ref{sec: semi direct}.

Even though we prove concrete and unconditional results in this paper towards the Defect Conjecture and can also calculate Iwasawa invariants, it is pertinent to acknowledge that we were able to predict the correct shape of these results by performing extensive SAGE calculations.
For the benefit of the reader and researchers intending to extend these calculations, we have included the SAGE code in the Appendix. 

\subsection{Outlook}
We have calculated the Iwasawa invariants for cyclic digraphs but have refrained from performing these calculations for more general Cayley graphs of (finite) abelian groups in this article.
Although the calculations are not expected to be difficult, they are tedious and involve careful bookkeeping when working in full generality.
This is being worked out in a separate ongoing project by one of the authors.
In the future, we intend to extend these results to Cayley graphs of other families of non-abelian groups.
Our results can be extended to unramified and ramified $\Z^d_p$-extensions.

\subsection{Organization}
This paper has been written for a reader who is not an expert in Iwasawa theory and we have also tried to make it as self-contained as possible.
In Section~\ref{sec: preliminaries}, we record and explain the basics of Cayley graphs, Galois covers of graphs, construction of $\Z_p$-extensions of base graphs, foundational results in Iwasawa theory, and preliminaries on Bowen--Franks groups. 
We provide proofs of results which are possibly known to experts but may not have been written down explicitly.
We include our main results in Section~\ref{classesofexamples} which itself is divided into four subsections.
We start with Cayley graphs of cyclic groups (these are cyclic digraphs) as a warm-up exercise; we prove the Defect Conjecture and calculate the Iwasawa invariants.
Next, we prove the results for dihedral groups in detail and then extend the results to $\Z/p\Z \rtimes \Z/(p-1)\Z$.
We prove the Defect Conjecture for all abelian groups in Section~\ref{sec: abelian}.
The SAGE code is provided in Appendix~\ref{sage code}.

\subsection*{Acknowledgements}
This project started at the Women in Numbers 7 workshop hosted at the BIRS, Banff in June 2026.
We thank the organizers for making this collaboration possible, the sponsors for their financial support, and BIRS for providing a stimulating environment during this research.
DK acknowledges the support of NSERC Discovery grants RGPIN-2026-07384 and DGECR-2026-00254.
KM's trip to Banff was partially supported by a PIMS--Simons Travel Award.
CS acknowledges partial support from the Cecilia S. Miranda and Henry A. Miranda, Jr. Mathematics Fellowship.
\section{Preliminaries}
\label{sec: preliminaries}

\subsection{Directed Graphs and their Galois Coverings}

In this section, we review the basic theory of directed graphs, morphisms, and coverings following \cite[Section~2]{lei2026iwasawa}.
\begin{definition}
A \emph{directed graph} (also called a \emph{digraph}) $X$ consists of a set of vertices $V(X)$ and a set of (directed or oriented) edges $E(X)$, together with an incidence map
\begin{align*}
\operatorname{inc}:E(X) &\longrightarrow V(X) \times V(X) \\
e & \longmapsto \operatorname{inc}(e) = (o(e),t(e)),
\end{align*}
where $o(e)$ is the origin (or source) and $t(e)$ is the terminus (or target) of the directed edge $e$, respectively.  

Given a vertex $v\in V(X)$, define the sets
\begin{align*}
    \In(v)&=\{e\in E(X):t(e)=v\},\\
    \Out(v)&=\{e\in E(X):o(e)=v\}. 
\end{align*}
The \emph{in-degree} and \emph{out-degree} of $v$ are defined as the cardinality of $\In(v)$ and $\Out(v)$, respectively.
\end{definition}

\begin{definition}
Let $X$ and $Y$ be two digraphs.
\begin{enumerate}[label = \textup{(}\roman*\textup{)}]
\item A \emph{morphism} of digraphs $f:Y \to X$ is a map that induces two maps 
\[ V(Y) \longrightarrow V(X) \text{ and } E(Y) \longrightarrow E(X)
\]
that commute with the incidence map.
\item  A bijective morphism between $X$ and $Y$ is called an \emph{isomorphism}, and is denoted by $X\cong Y$.
\item If there is a surjective digraph morphism $f:Y \to X$ that induces bijections from $\In(v)$ (resp. $\Out(v)$) to $\In(f(v))$ (resp. $\Out(f(v))$) for every $v\in V(Y)$, then $Y/X$ is called a \emph{covering}.
 \end{enumerate}
\end{definition}

\begin{definition}
Let $Y/X$ be a covering of directed graphs with the projection map $\pi:Y\to X$.
\begin{enumerate}[label = \textup{(}\roman*\textup{)}]
\item The group of \emph{deck transformations} of $Y/X$, denoted by $\Deck(Y/X)$, is the group of digraph automorphisms $\sigma:Y\to Y$ such that $\pi\circ\sigma=\pi$.
\item A covering $Y/X$ is said to be \emph{$d$-sheeted} if $d$ is a positive integer such that each element of $V(X)$ has $d$ pre-images in $V(Y)$.
\item The covering $Y/X$ is called \emph{Galois}  if there exists an integer $d$ such that $Y/X$ is a $d$-sheeted covering and $\vert \textup{Deck}(Y/X)\vert =d$.
\end{enumerate}
\end{definition}

We explain this notion via an example and a non-example below.

\begin{example}
Consider the base graph $X=C_3$ which is the (directed) cycle graph with three vertices $v_1 \to v_2 \to v_3 \to v_1$. 
Let $Y=C_6$ be the (directed) cycle graph with six vertices $w_1\to w_2 \to w_3 \to w_4 \to w_5 \to w_6 \to w_1$.
We can define the map $\pi: Y\to X$ by $\pi(w_i) = v_{i \pmod{3}}$.
Then we have $\pi^{-1}(v_1) = \{w_1, w_4\}$, $\pi^{-1}(v_2) = \{w_2, w_5\}$, and $\pi^{-1}(v_3)=\{w_0, w_3\}$.
In particular, $Y$ is a 2-sheeted cover of $X$.
A deck transformation is completely determined by where it sends its vertices.
If we consider $\sigma: Y\to Y$ defined via $w_i \mapsto w_{i+3 \pmod{6}}$, then $\pi \circ \sigma = \pi$.
The other possibility of $\sigma$ is the identity map.
Thus $\Deck(Y/X)\simeq \Z/2\Z$.
Thus, $Y/X$ is a Galois covering. 
\end{example}

\begin{nexample}
Consider the following covering

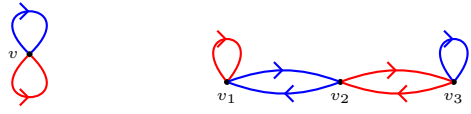
\begin{center}
\begin{tikzpicture}[scale=0.75]

\node[inner sep=0pt, label = left:\tiny{$v$}] (A) at (0,0) {};

\fill (0,0) circle (1.5pt);

\draw[thick, mid arrow, blue]
(A) .. controls (-1,1) and (1,1) .. 
node[midway, left] {} (A);

\draw[thick, mid arrow, red]
(A) .. controls (-1,-1) and (1,-1) .. 
node[midway, right] {} (A);

\end{tikzpicture}
\hspace{1 cm}
\begin{tikzpicture}[scale=0.75]

\node[inner sep=0pt, label = below:\tiny{$v_1$}] (A1) at (-2,0) {};
\node[inner sep=0pt, label = below:\tiny{$v_2$}] (A2) at (0,0) {};
\node[inner sep=0pt, label = below:\tiny{$v_3$}] (A3) at (2,0) {};

\fill (-2,0) circle (1.5pt);
\fill (0,0) circle (1.5pt);
\fill (2,0) circle (1.5pt);

\draw[thick, mid arrow, red]
(A1) .. controls (-2.8,1) and (-1.2,1) .. (A1);

\draw[thick, mid arrow,blue]
(A1) to[bend left=20] (A2);

\draw[thick, mid arrow, blue]
(A2) to[bend left=20] (A1);

\draw[thick, mid arrow, red]
(A2) to[bend left=20] (A3);

\draw[thick, mid arrow, red]
(A3) to[bend left=20] (A2);

\draw[thick, mid arrow, blue]
(A3) .. controls (1.2,1) and (2.8,1) .. (A3);
\end{tikzpicture}
\end{center}
Here $X$ is the graph with a single vertex $v$ and two edges.
The graph $Y$ has three vertices and corresponding edges.
Note that the projection map $\pi$ sends every red (resp. blue) segment or loop to the red (resp. blue) loop respecting the given orientation.
Thus, $Y$ is a $3$-sheeted cover of $X$.

The three vertices in $Y$ are different from each other.
Indeed, $v_1$ has one red loop and two blue segments, $v_2$ has two blue segments and two red segments, and $v_3$ has one blue loop and two red segments. 
A deck automorphism must preserve these three vertices, but this can be done only by the identity.
\end{nexample}

\subsection{Cayley Graphs}

A Cayley graph encodes the abstract structure of a group.
Let $G$ be a group and $S$ be a generating set of $G$.
The Cayley graph $X=X(G,S)$ is a directed graph constructed as follows:
\begin{itemize}
    \item Each element $g\in G$ is assigned a vertex, and the vertex set $V(X)$ of the graph is identified with $G$.
    \item For every $g\in G$ and $s\in S$, there is a directed edge from the vertex corresponding to $g$ to the one corresponding to $g\cdot s$.
\end{itemize}
The Cayley graph $X=X(G,S)$ depends in an essential way on the choice of the set $S$ of generators. 

\begin{example}
The dihedral group $D_{2n}$ (of order $2n$) has two different presentations, namely
\begin{align*}
    \text{standard presentation given by } & \langle \sigma, \tau \mid \sigma^n = \tau^2 = e, \ \tau \sigma \tau = \sigma^{-1}\rangle \\ 
    \text{Coxeter presentation given by } & \langle s,t \mid s^2 = t^2 = (st)^n = e\rangle.
\end{align*}
In the standard presentation, $\sigma$ denotes the rotation by $2\pi/n$ radians and $\tau$ denotes the reflection.
But, a dihedral group can also be generated by two reflections.
Indeed, if $s$ and $t$ are two reflections of an $n$-gon across adjacent axes of symmetry (i.e., axes incident at $\pi/n$ radians), then $st$ is a rotation by $2\pi/n$.
Here, we take $s = \sigma\tau$ and $t = \tau$ in $D_{2n}$; and it is clear that $st = (\sigma\tau)\tau = \sigma$.
Figures~\ref{fig 1} and \ref{fig 2} show the Cayley graphs of $D_6$ with respect to these two presentations.

\begin{center}
\begin{figure}[h!]
\begin{tikzpicture} [%
    nd/.style = {circle,fill=black,text=white,inner sep=1pt},
    tn/.style = {node distance=1pt},
    redarrow/.style={->, red, fill=none,>=stealth},
    blueline/.style={->,blue,fill=none}]


\def\d{1.5} 

\node[nd] (it)  at (0, 1) {};
\node[nd] (ibl) at ( {-\d * cos(60)}, {(-\d * sin(60)) + 1} ) {};
\node[nd] (ibr) at ( {\d * cos(60)},  {(-\d * sin(60)) + 1}) {};

\node[nd] (ot)  at (0, 2*\d) {};
\node[nd] (obl) at ( {-2 * \d * cos(30)}, {-2 * \d * sin(30)} ) {};
\node[nd] (obr) at ( {2 * \d * cos(30)},  {-2 * \d * sin(30)} ) {};

    \draw[redarrow] (ot) -- (obl);
    \draw[redarrow] (obl) -- (obr);
    \draw[redarrow] (obr) -- (ot); 
    \draw[redarrow] (it) -- (ibr);
    \draw[redarrow] (ibr) -- (ibl);
    \draw[redarrow] (ibl) -- (it); 

    \draw[blueline] (ibl) to[bend left=15] (obl);
    \draw[blueline] (obl) to[bend left=15] (ibl);
    \draw[blueline] (it) to[bend left=15] (ot);
    \draw[blueline] (ibr) to[bend left=15] (obr);
    \draw[blueline] (ot) to[bend left=15] (it);
    \draw[blueline] (obr) to[bend left=15] (ibr);
    
    \node[tn] [below=of it] {\tiny{$\sigma$}};
    \node[tn] [below=of ibr] {\tiny{$\sigma^2$}};
    \node[tn] [below=of ibl] {\tiny{$e$}};

    \node[tn] [below left=of obl] {\tiny{$\tau$}};
    \node[tn] [below=of obr] {\tiny{$\tau \sigma = \sigma^2 \tau$}};
    \node[tn] [above=of ot] {\tiny{$\sigma \tau =  \tau \sigma^2$}};
    
\end{tikzpicture}
\caption{Cayley graph of $D_{6}=  \langle \sigma, \tau \mid \sigma^3 = \tau^2 = e, \tau \sigma \tau = \sigma^{-1}\rangle$}
\label{fig 1}
\end{figure}
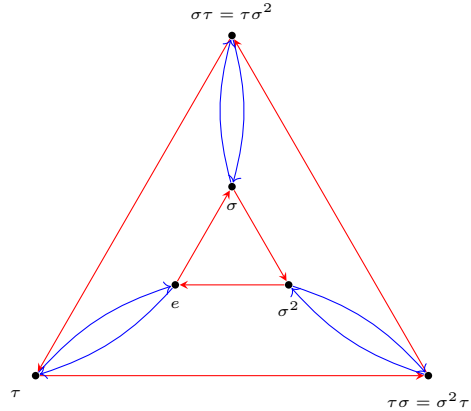
\end{center}

\begin{center}
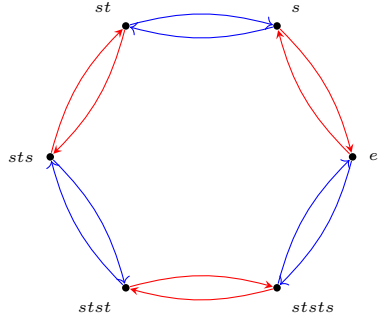
\begin{figure}[h!]
\begin{tikzpicture} [%
    nd/.style = {circle,fill=black,text=white,inner sep=1pt},
    tn/.style = {node distance=1pt},
    redarrow/.style={->, red, fill=none,>=stealth},
    blueline/.style={->,blue,fill=none}]
    
\def\r{2.0} 

\node[nd] (v0) at (0:\r) {};
\node[nd] (v1) at (60:\r) {};
\node[nd] (v2) at (120:\r) {};
\node[nd] (v3) at (180:\r) {};
\node[nd] (v4) at (240:\r) {};
\node[nd] (v5) at (300:\r) {};

\node[tn] [right=of v0] {\tiny{$e$}};
\node[tn] [above right=of v1] {\tiny{$s$}};
\node[tn] [above left=of v2] {\tiny{$st$}};
\node[tn] [left=of v3] {\tiny{$sts$}};
\node[tn] [below left=of v4] {\tiny{$stst$}};
\node[tn] [below right=of v5] {\tiny{$ststs$}};

    \draw[redarrow] (v0) to[bend left=15] (v1);
    \draw[redarrow] (v1) to[bend left=15] (v0);
    \draw[blueline] (v1) to[bend left=15] (v2); 
    \draw[blueline] (v2) to[bend left=15] (v1);
    \draw[redarrow] (v2) to[bend left=15] (v3);
    \draw[redarrow] (v3) to[bend left=15] (v2);
    \draw[blueline] (v3) to[bend left=15] (v4); 
    \draw[blueline] (v4) to[bend left=15] (v3);
    \draw[redarrow] (v4) to[bend left=15] (v5);
    \draw[redarrow] (v5) to[bend left=15] (v4);
    \draw[blueline] (v0) to[bend left=15] (v5); 
    \draw[blueline] (v5) to[bend left=15] (v0);

\end{tikzpicture}
\caption{Cayley graph of $D_{6}=  \langle s,t \mid s^2 = t^2 = (st)^3 = e\rangle$}
\label{fig 2}
\end{figure}
\end{center}
\end{example}

\subsubsection{Representation theory}
\label{cayley and rep theory}
Given a finite group $G$ and a generating set $S$, we want to understand the precise relationship between the eigenvalues of the adjacency matrix $A_X$ of $X(G,S)$ and the irreducible representations of $G$.
Recall the following definition
\begin{definition}
Let $X$ be a finite connected graph with $n$ vertices.
The \emph{adjacency matrix} $A_X = \left[a_{ij}\right]_{1\leq i,j\leq n}$ is an $n\times n$ matrix defined as 
    \[
    a_{ij} \coloneqq 
        \begin{cases}
            \operatorname{mult}(v_i,v_j), &\text{ if } v_i  \text{ and } v_j \text{ are adjacent},\\
            0, & \text{ otherwise},
        \end{cases}
    \]
    where $\operatorname{mult}(v_i,v_j)$ denotes the number of edges from $v_i$ to $v_j$.
\end{definition}    
We start with the complex vector space $\C[G]$ with basis $\{e_g : g\in G\}$.
The regular representation
\[
\rho: G \longrightarrow \GL(\C[G]) \text{ given by } \rho(h)e_g = e_{gh}
\]
extends to 
\[
\rho(h)\left( \sum_{g\in G} c_g e_g\right) = \sum_{g\in G} c_g e_{gh}.
\]
Consider the adjacency operator $A_X$ which is defined on the basis vector $e_g$ via
\[
A_Xe_g = \sum_{x : g\to x} e_x
\]
But in a (directed) Cayley graph the `outgoing' neighbours of $g$ are $\{gs : s\in S\}$.
This means we can define the adjacency operator
\[
A_Xe_g = \sum_{s \in S} e_{gs} = \sum_{s\in S} \rho(s)e_g.
\]
Therefore, the adjacency operator
\[
A_X = \sum_{s\in S} \rho(s)
\]
By Maschke's Theorem we can decompose 
\[
\rho = \bigoplus_{\rho_i \text{ irreducible}} \dim(\rho_i) \rho_i.
\]
Writing $M_{\rho_i}$ to denote the block matrix corresponding to the representation $\rho_i$, we can write
\begin{equation}
\label{adjacency}
A_X \sim \bigoplus_{\rho_i \text{ irreducible}} \dim(\rho_i) M_{\rho_i}
\end{equation}
where the right hand side is a matrix obtained by placing each block matrix $M_{\rho_i}$ along the diagonal.

\subsection{Iwasawa Theory}
\label{Iwasawa theory}
Fix a prime $p$.
For a non-negative integer $m$, write $\zeta_{p^m}$ to denote a primitive $p^m$-th root of unity.
The \emph{cyclotomic} $\Z_p$-extension of $\Q$ is the maximal totally real pro-$p$ subfield of $\bigcup_m \Q(\zeta_{p^m})$, denoted by $\Q_{\cyc}$.
In particular, there is a tower
\[
\Q = \Q_{(0)} \subset \Q_{(1)} \subset \ldots \subset \Q_{(m)} \subset \ldots \subset \Q_{\cyc}
\]
where each $\Q_{(m)}$ is the maximal totally real subfield of degree a power of $p$ in $\Q(\zeta_{p^{m+1}})$ (resp. $\Q(\zeta_{p^{m+2}})$) when $p$ is odd (resp. even) with $\Gal(\Q_{(m)}/\Q)\simeq \Z/p^m\Z$.
By infinite Galois theory,
\[
\Gamma = \Gal(\Q_{\cyc}/\Q) \simeq \varprojlim_m \Z/p^m\Z = \Z_p.
\]
When $K$ is any number field, the cyclotomic $\Z_p$-extension of $K$ is $K_{\cyc} = K \cdot \Q_{\cyc}$; i.e., it is the compositum of $K$ with $\Q_{\cyc}$.
As before, we write $K_{(m)}$ to denote the $m$-th layer of the cyclotomic extension of $K$ and write $\Gamma_m = \Gal(K_{(m)}/K) \simeq \Z/p^m \Z$.
By definition, $\Gal(K_{\cyc}/K)\simeq \Z_p$ as well.
In general, when the number field $K$ is not totally real, there are infinitely many $\Z_p$-extensions of $K$.

A main focus in the study of classical Iwasawa theory is to study the growth of the $p$-part of the class group in each layer of a $\Z_p$-extension of a given base field.
In the 1950's Iwasawa proved that for any $\Z_p$-extension of a number field $K$, there exist non-negative integers $\lambda,\mu$ and an integer $\nu$ (all independent of $m$) such that 
\begin{equation}
\label{Iwasawa formula}
\val_p(|\operatorname{Cl}(K_{(m)})|) = \mu p^m + \lambda m +\nu \text{ when } m\gg0.
\end{equation}
The constants $\lambda$, $\mu$, and $\nu$ depend on the prime $p$, the base field, and on the $\Z_p$-extension.


\subsubsection{Constant \texorpdfstring{$\Z_p$}{}-towers of (di)graphs}

The following section summarizes the construction of constant Galois $\Z_p$-towers of digraphs appearing in \cite{lei2026iwasawa}.

\begin{definition} 
Fix a prime $p$ and a finite connected digraph $X=X_0$.
The sequence of Galois covers 
\[
X=X_0 \longleftarrow X_1 \longleftarrow X_2 \longleftarrow \ldots \longleftarrow X_m \longleftarrow \ldots
\]
is a $\Z_p$-tower, if each composed graph covering $X_m \rightarrow X_0$ has Galois group isomorphic to $\Z/p^m\Z$.
\end{definition}

It follows from infinite Galois theory that $\Gamma = \varprojlim_n\Gal(X_n/X) \simeq \Z_p$ and $\Gal(X_m/X) = \Gamma/\Gamma^{p^m} \simeq \Z/p^m\Z$. 
Next, we construct a $\mathcal{G}$-tower of digraphs for a group $\mathcal{G}$ using a voltage assignment on the edges of a digraph.

\begin{definition} 
Let $X$ be a digraph and $\mathcal{G}$ be a group.
\begin{enumerate}[label = \textup{(}\roman*\textup{)}]
\item A function $\alpha: E(X) \to \mathcal{G}$ is called a \emph{$\mathcal{G}$-valued voltage assignment on $X$}.
\item When $\mathcal{G}$ is finite, the \emph{derived digraph} $X(\mathcal{G}, \alpha)$ has vertices $V(X) \times \mathcal{G}$ and directed edges $E(X) \times \mathcal{G}$.
Moreover, for all $g\in \mathcal{G}$ and $e\in E(X)$ with $\operatorname{inc}(e) = (o, t)$ the incidence map of the derived graph is
\[
\operatorname{inc}(e,g) = ((o,g) , (t, g\cdot \alpha(e)))
\]
\end{enumerate}
\end{definition}

When a $\Z_p$-tower of $X$ arises from a constant voltage assignment $\alpha$, i.e., the image of $\alpha$ consists of a single element and assigns this element to every edge, we refer to it as a \emph{constant $\Z_p$-tower}.
In this case, each finite layer $X_m$ is viewed as a derived graph $X(\Gamma/\Gamma^{p^m}, \alpha_m)$ as follows: let $\theta$ be the topological generator of $\Gamma = \Gal(X_\infty/X) \simeq \Z_p$, the $\Z/p^m\Z$-valued voltage assignment $\alpha_m = \pi_m \circ \alpha$ is given by
\begin{alignat*}{2}
    \alpha_m\colon E(X) & \xrightarrow{\alpha} \Z_p && \xrightarrow{\pi_m} \Z/p^m\Z\\
    e &\mapsto \theta && \longmapsto \theta_m.
\end{alignat*}
Note that the graph $X_m$ is not necessarily connected.

But along the tower, the number of connected components stays uniformly bounded.
Let $m_0$ be the \textit{minimal index} where the number of connected components stabilizes.
Then $X_{m_0}$ consists of $p^{m_0}$ copies of $X$.
Fix one of these connected components and denote it by $X'_0$.
For every $m>m_0$ there is a unique connected component $X'_m$  of $X_m$ lying over $X'_{m_0}$.
In this case $X'_m/X'_{m_0}$ is Galois with Galois group $\Z/p^{m-m_0}\Z$.
Thus, we can just skip the first $m_0$ steps of the tower and restrict to a tower above one connected component.
As the base graph is isomorphic to $X$ one usually calls this restricted tower the constant $\Z_p$-tower of $X$.

\begin{example}
Let $X$ be the cycle graph $C_3$ with three vertices and three edges.
Let $p=3$ and consider the constant voltage assignment assigning the same topological generator $\tau$ to every edge.
The graph $X_1=X(\Z/3\Z,\alpha_1)$ consists of the following three cycles:
\begin{align*}
(v_0,1)\to(v_1,\tau)\to(v_2,\tau^2)\to(v_0,1)\\
(v_0,\tau)\to (v_1,\tau^2)\to(v_3,1)\to (v_0,\tau)\\
(v_0,\tau^2)\to(v_2,1)\to(v_3,\tau)\to (v_0,\tau^2).
\end{align*}
Thus, it consists indeed of three copies of the base graph. The graph $X_2=X(\Z/9\Z,\alpha_2)$ consist again of three cycles but of length $9$. One of which is
\[
(v_0,1) \to (v_1,\tau) \to(v_2,\tau^2) \to(v_0,\tau^3)\to (v_1,\tau^4) \to(v_2,\tau^5)\to(v_0,\tau^6)\to(v_1,\tau^7)\to(v_2,\tau^8).
\]
The graph $X_n=X(\Z/3^n\Z,\alpha_n)$ will consist of $3$ cycles of length $3^{n}$.
More generally, if $X$ is a cycle graph $C_n$ and we fix a prime $p$, then $m_0 = \val_p(n)$.
\end{example}

We explain this with the help of a diagram.
First we start with a directed graph which has three vertices and consider the case that $p=2$.
We draw the first two layers.

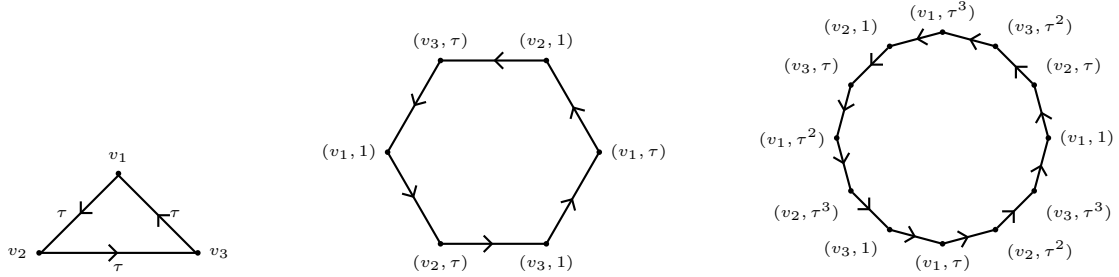
\begin{center}
\begin{tikzpicture}[scale=0.7]
\node[inner sep=0pt, label = above:\tiny{$v_1$}] (A) at (0,1.5) {};
\node[inner sep=0pt, label = left:\tiny{$v_2$}] (B) at (-1.5,0) {}; 
\node[inner sep=0pt, label = right:\tiny{$v_3$}] (C) at (1.5,0) {}; 
\node[inner sep=0pt, label = left:\tiny{$\tau$}] (T) at (-0.75,0.75) {};
\node[inner sep=0pt, label = right:\tiny{$\tau$}] (T) at (0.75,0.75) {};
\node[inner sep=0pt, label = below:\tiny{$\tau$}] (T) at (0,0) {};

\fill (0,1.5) circle (1.5pt);
\fill (1.5,0) circle (1.5pt);
\fill (-1.5,0) circle (1.5pt);

\draw[thick, mid arrow] (A.south) to (B.east);
\draw[thick, mid arrow] (C.west) to (A.south);
\draw[thick, mid arrow] (B.east) to (C.west);
\end{tikzpicture}
\hspace{0.75cm}
\begin{tikzpicture}[scale=0.7]
\node[inner sep=0pt, label = left:\tiny{$(v_1,1)$}] (A1) at (-2,0) {};
\node[inner sep=0pt, label = right:\tiny{$(v_1, \tau)$}] (B3) at (2,0) {}; 
\node[inner sep=0pt, label = above:\tiny{$(v_3,\tau)$}] (B2) at (2*-0.5,2*0.866) {};
\node[inner sep=0pt, label = above:\tiny{$(v_2, 1)$}] (A2) at (2*0.5,2*0.866) {}; 
\node[inner sep=0pt, label = below:\tiny{$(v_2,\tau)$}] (B1) at (2*-0.5,2*-0.866) {};
\node[inner sep=0pt, label = below:\tiny{$(v_3,1)$}] (A3) at (2*0.5,2*-0.866) {}; 

\fill (2,0) circle (1.5pt);
\fill (-2,0) circle (1.5pt);
\fill (2*-0.5,2*0.866) circle (1.5pt);
\fill (2*-0.5,-2*0.866) circle (1.5pt);
\fill (2*0.5,2*0.866) circle (1.5pt);
\fill (2*0.5,-2*0.866) circle (1.5pt);

\draw[thick, mid arrow] (A1) to (B1);
\draw[thick, mid arrow] (B2) to (A1);
\draw[thick, mid arrow] (A2) to (B2);
\draw[thick, mid arrow] (B3) to (A2);
\draw[thick, mid arrow] (B1) to (A3);
\draw[thick, mid arrow] (A3) to (B3);
\end{tikzpicture} 
\hspace{0.75cm}
\begin{tikzpicture}[scale=0.7]
\node[inner sep=0pt, label = right:\tiny{$(v_1,1)$}]   (A1)  at (2,0) {};
\node[inner sep=0pt, label = above right:\tiny{$(v_2,\tau)$}]  (A2)  at (2*0.866,2*0.5) {};
\node[inner sep=0pt, label = above right:\tiny{$(v_3,\tau^2)$}]   (A3)  at (2*0.5,2*0.866) {};
\node[inner sep=0pt, label = above:\tiny{$(v_1,\tau^3)$}]  (A4)  at (0,2) {};
\node[inner sep=0pt, label = above left:\tiny{$(v_2,1)$}]   (A5)  at (-2*0.5,2*0.866) {};
\node[inner sep=0pt, label = above left:\tiny{$(v_3,\tau)$}]  (A6)  at (-2*0.866,2*0.5) {};

\node[inner sep=0pt, label = left:\tiny{$(v_1,\tau^2)$}]   (A7)  at (-2,0) {};
\node[inner sep=0pt, label = below left:\tiny{$(v_2,\tau^3)$}]  (A8)  at (-2*0.866,-2*0.5) {};
\node[inner sep=0pt, label = below left:\tiny{$(v_3,1)$}]   (A9)  at (-2*0.5,-2*0.866) {};
\node[inner sep=0pt, label = below:\tiny{$(v_1,\tau)$}] (A10) at (0,-2) {};
\node[inner sep=0pt, label = below right:\tiny{$(v_2,\tau^2)$}]  (A11) at (2*0.5,-2*0.866) {};
\node[inner sep=0pt, label = below right:\tiny{$(v_3,\tau^3)$}] (A12) at (2*0.866,-2*0.5) {};

\fill (2,0) circle (1.5pt);
\fill (2*0.866,2*0.5) circle (1.5pt);
\fill (2*0.5,2*0.866) circle (1.5pt);
\fill (0,2) circle (1.5pt);
\fill (-2*0.5,2*0.866) circle (1.5pt);
\fill (-2*0.866,2*0.5) circle (1.5pt);

\fill (-2,0) circle (1.5pt);
\fill (-2*0.866,-2*0.5) circle (1.5pt);
\fill (-2*0.5,-2*0.866) circle (1.5pt);
\fill (0,-2) circle (1.5pt);
\fill (2*0.5,-2*0.866) circle (1.5pt);
\fill (2*0.866,-2*0.5) circle (1.5pt);

\draw[thick, mid arrow] (A1) to (A2);
\draw[thick, mid arrow] (A2) to (A3);
\draw[thick, mid arrow] (A3) to (A4);
\draw[thick, mid arrow] (A4) to (A5);
\draw[thick, mid arrow] (A5) to (A6);
\draw[thick, mid arrow] (A6) to (A7);
\draw[thick, mid arrow] (A7) to (A8);
\draw[thick, mid arrow] (A8) to (A9);
\draw[thick, mid arrow] (A9) to (A10);
\draw[thick, mid arrow] (A10) to (A11);
\draw[thick, mid arrow] (A11) to (A12);
\draw[thick, mid arrow] (A12) to (A1);
\end{tikzpicture}
\end{center}
We now consider the case that $p=3$.
We once again draw the first two layers.
In this case we observe that each layer (above the base graph) is disconnected.

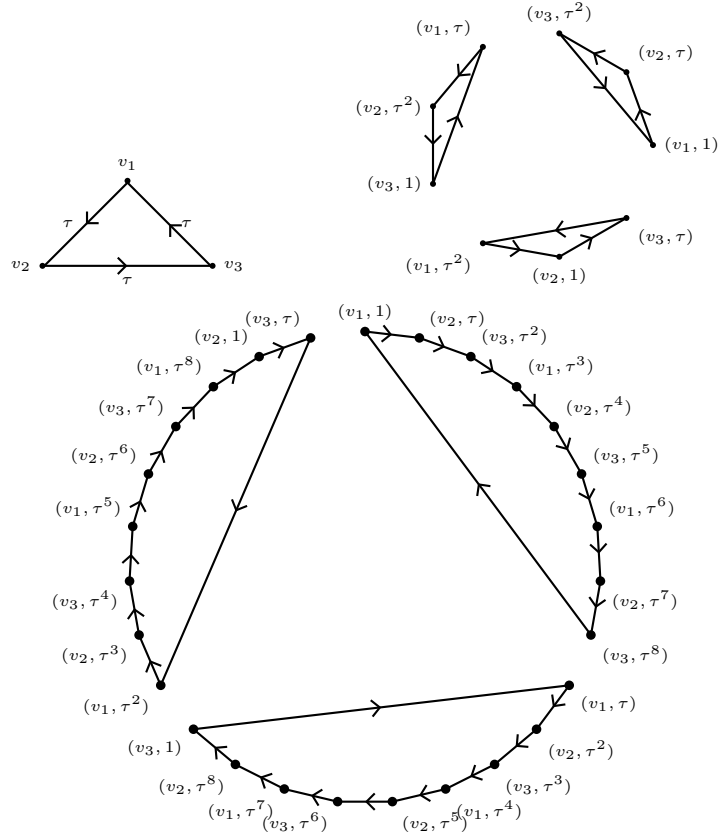
\begin{center}
\begin{tikzpicture}[scale=0.75]
\node[inner sep=0pt, label = above:\tiny{$v_1$}] (A) at (0,1.5) {};
\node[inner sep=0pt, label = left:\tiny{$v_2$}] (B) at (-1.5,0) {}; 
\node[inner sep=0pt, label = right:\tiny{$v_3$}] (C) at (1.5,0) {}; 
\node[inner sep=0pt, label = left:\tiny{$\tau$}] (T) at (-0.75,0.75) {};
\node[inner sep=0pt, label = right:\tiny{$\tau$}] (T) at (0.75,0.75) {};
\node[inner sep=0pt, label = below:\tiny{$\tau$}] (T) at (0,0) {};

\fill (0,1.5) circle (1.5pt);
\fill (1.5,0) circle (1.5pt);
\fill (-1.5,0) circle (1.5pt);

\draw[thick, mid arrow] (A.south) to (B.east);
\draw[thick, mid arrow] (C.west) to (A.south);
\draw[thick, mid arrow] (B.east) to (C.west);
\end{tikzpicture}
\hspace{1cm}
\begin{tikzpicture}[scale=0.75]

\node[inner sep=0pt, label = right:\tiny{$(v_1,1)$}] (A1) at (2,0) {};
\node[inner sep=0pt, label = above right:\tiny{$(v_2,\tau)$}] (A2) at (2*0.766,2*0.643) {};
\node[inner sep=0pt, label = above:\tiny{$(v_3,\tau^2)$}] (A3) at (2*0.174,2*0.985) {};
\node[inner sep=0pt, label = above left:\tiny{$(v_1,\tau)$}] (A4) at (-2*0.5,2*0.866) {};
\node[inner sep=0pt, label = left:\tiny{$(v_2,\tau^2)$}] (A5) at (-2*0.94,2*0.342) {};

\node[inner sep=0pt, label = left:\tiny{$(v_3,1)$}] (A6) at (-2*0.94,-2*0.342) {};
\node[inner sep=0pt, label = below left:\tiny{$(v_1,\tau^2)$}] (A7) at (-2*0.5,-2*0.866) {};
\node[inner sep=0pt, label = below:\tiny{$(v_2,1)$}] (A8) at (2*0.174,-2*0.985) {};
\node[inner sep=0pt, label = below right:\tiny{$(v_3,\tau)$}] (A9) at (2*0.766,-2*0.643) {};

\fill (2,0) circle (1.5pt);
\fill (2*0.766,2*0.643) circle (1.5pt);
\fill (2*0.174,2*0.985) circle (1.5pt);
\fill (-2*0.5,2*0.866) circle (1.5pt);
\fill (-2*0.94,2*0.342) circle (1.5pt);
\fill (-2*0.94,-2*0.342) circle (1.5pt);
\fill (-2*0.5,-2*0.866) circle (1.5pt);
\fill (2*0.174,-2*0.985) circle (1.5pt);
\fill (2*0.766,-2*0.643) circle (1.5pt);

\draw[thick, mid arrow] (A1) to (A2);
\draw[thick, mid arrow] (A2) to (A3);
\draw[thick, mid arrow] (A3) to (A1);

\draw[thick, mid arrow] (A4) to (A5);
\draw[thick, mid arrow] (A5) to (A6);
\draw[thick, mid arrow] (A6) to (A4);

\draw[thick, mid arrow] (A7) to (A8);
\draw[thick, mid arrow] (A8) to (A9);
\draw[thick, mid arrow] (A9) to (A7);

\end{tikzpicture}
\hspace{1cm}
\begin{tikzpicture}[scale=1.2]

\foreach \i/\lab in {
1/{(v_1,1)},
2/{(v_2,\tau)},
3/{(v_3,\tau^2)},
4/{(v_1,\tau^3)},
5/{(v_2,\tau^4)},
6/{(v_3,\tau^5)},
7/{(v_1,\tau^6)},
8/{(v_2,\tau^7)},
9/{(v_3,\tau^8)},
10/{(v_1,\tau)},
11/{(v_2,\tau^2)},
12/{(v_3,\tau^3)},
13/{(v_1,\tau^4)},
14/{(v_2,\tau^5)},
15/{(v_3,\tau^6)},
16/{(v_1,\tau^7)},
17/{(v_2,\tau^8)},
18/{(v_3,1)},
19/{(v_1,\tau^2)},
20/{(v_2,\tau^3)},
21/{(v_3,\tau^4)},
22/{(v_1,\tau^5)},
23/{(v_2,\tau^6)},
24/{(v_3,\tau^7)},
25/{(v_1,\tau^8)},
26/{(v_2,1)},
27/{(v_3,\tau)}
}
{
    \pgfmathsetmacro{\ang}{90 - 360*(\i-1)/27}
    \node[inner sep=0pt, label={\ang:\tiny{$\lab$}}] (A\i) at (\ang:2.6) {};
    \fill (\ang:2.6) circle (1.5pt);
}

\draw[thick, mid arrow] (A1) to (A2);
\draw[thick, mid arrow] (A2) to (A3);
\draw[thick, mid arrow] (A3) to (A4);
\draw[thick, mid arrow] (A4) to (A5);
\draw[thick, mid arrow] (A5) to (A6);
\draw[thick, mid arrow] (A6) to (A7);
\draw[thick, mid arrow] (A7) to (A8);
\draw[thick, mid arrow] (A8) to (A9);
\draw[thick, mid arrow] (A9) to (A1);

\draw[thick, mid arrow] (A10) to (A11);
\draw[thick, mid arrow] (A11) to (A12);
\draw[thick, mid arrow] (A12) to (A13);
\draw[thick, mid arrow] (A13) to (A14);
\draw[thick, mid arrow] (A14) to (A15);
\draw[thick, mid arrow] (A15) to (A16);
\draw[thick, mid arrow] (A16) to (A17);
\draw[thick, mid arrow] (A17) to (A18);
\draw[thick, mid arrow] (A18) to (A10);

\draw[thick, mid arrow] (A19) to (A20);
\draw[thick, mid arrow] (A20) to (A21);
\draw[thick, mid arrow] (A21) to (A22);
\draw[thick, mid arrow] (A22) to (A23);
\draw[thick, mid arrow] (A23) to (A24);
\draw[thick, mid arrow] (A24) to (A25);
\draw[thick, mid arrow] (A25) to (A26);
\draw[thick, mid arrow] (A26) to (A27);
\draw[thick, mid arrow] (A27) to (A19);

\end{tikzpicture}
\end{center}

\subsubsection{Iwasawa Algebra and the Structure Theorem}

The \emph{Iwasawa algebra} $\Lambda=\Lambda(\Gamma)$ is the completed group algebra $\Z_p\llbracket \Gamma \rrbracket :=\varprojlim_m \Z_p[\Gamma/\Gamma^{p^m}]$.
Fix a topological generator $\theta$ of $\Gamma$; this gives an isomorphism of rings 
\begin{align*}
\Lambda &\xrightarrow{\sim} \Z_p\llbracket T\rrbracket \\
\theta & \mapsto 1+T.
\end{align*}

Let $M$ be a finitely generated torsion $\Lambda$-module.
The \emph{Structure Theorem of $\Lambda$-modules} asserts \cite[Theorem~13.12]{Was97} that $M$ is pseudo-isomorphic to a finite direct sum of cyclic $\Lambda$-modules.
In other words, there is a homomorphism of $\Lambda$-modules
\[
M \longrightarrow \left(\bigoplus_{i=1}^s \Lambda/(p^{m_i})\right)\oplus \left(\bigoplus_{j=1}^t \Lambda/(f_j(T)) \right)
\]
with finite kernel and cokernel.
Here, $m_i>0$ and $f_j(x)$ is a distinguished polynomial (i.e. a monic polynomial with non-leading coefficients divisible by $p$).
The characteristic ideal of $M$ is (up to a unit) generated by the characteristic element/polynomial,
\[
f_{M}^{(p)}(T) := p^{\sum_{i} m_i} \prod_j f_j(T).
\]
The $\mu$-invariant of $M$ is defined as the power of $p$ in $f_{M}^{(p)}(T)$.
More explicitly,
\[
\mu(M) = \mu_p(M):=\begin{cases}0 & \textrm{ if } s=0\\
\sum_{i=1}^s m_i & \textrm{ if } s>0.
\end{cases}
\]
The $\lambda$-invariant of $M$ is the degree of the characteristic element, i.e.
\[
\lambda(M) = \lambda_p(M) := \sum_{j=1}^t \deg f_j(T).
\]

\begin{remark}
The invariants $\lambda,\mu$ appearing in Iwasawa's formula \eqref{Iwasawa formula} are indeed the invariants appearing in the above definition for the Galois module $M = \varprojlim_m \operatorname{Cl}(K_m)[p^\infty]$.
\end{remark}

\subsection{Bowen--Franks Groups}
\label{Bowen Franks}
In this section, $X$ denotes a finite connected graph with $n$ vertices.

\begin{definition}
Let $X$ be a finite connected graph with $n$ vertices.
\begin{enumerate}[label = \textup{(}\roman*\textup{)}]
\item A \emph{divisor} on a graph $X$ is an element of the free abelian group on its vertices $V(X)$,
    \[
        \Div(X) = \left\{\sum_{v\in V(X)} a_vv\mid v\in v(X), a_v\in \Z\right\}
    \]
    where $\displaystyle \sum_{v\in V(X)}a_vv$ is a formal sum of vertices of $X$.
\item The \emph{Bowen--Franks operator} on $\Div(X) \cong \Z^n$ is defined by the matrix $\operatorname{BF}_X = \Id - A_X$, where $A_X$ is the adjacency matrix.
    The \emph{Bowen--Franks group} is defined as 
    \[
        \BF(X) \coloneqq \Div(X)/\operatorname{BF_X\cdot}\Div(X) \cong \Z^n/(\Id-A_X)\Z^n.
    \]    
\end{enumerate}
\end{definition}

The matrix $\Id-A_X$ defines a homomorphism
\[
\Id-A_X:\Div(X)\longrightarrow \Div(X).
\]
By definition,
\[
\coker(\Id-A_X) = \frac{\Div(X)} {(\Id-A_X)\Div(X)}.
\]
So an alternative definition of the Bowen--Franks group is
\[
\BF(X) = \coker(\Id-A_X).
\]

\begin{definition}
Let $X$ be a graph with adjacency matrix $A_X$.
Set
\begin{align*}
a(X) &\coloneqq \text{ the algebraic multiplicity of 1 as an eigenvalue of } A_X \\
b(X) & \coloneqq \text{ the geometric multiplicity of 1 as an eigenvalue of } A_X = \rk_{\Z}(\BF(X)).
\end{align*}
Define $\delta(X) \coloneqq a(X) - b(X)$ as the \emph{defect} of $X$.
\end{definition}

\begin{remark}
In \cite{lei2026iwasawa}, the authors defined $a(X)$ in terms of the order of vanishing of a certain zeta function.
However, as explained in Example~7.2 of \textit{loc.~cit.} it follows from their proof of Lemma~4.7 that the two definitions are in fact equivalent.
\end{remark}

\begin{lemma}
\label{BFX infinite}
Let $X$ be a graph with adjacency matrix $A_X$.
The Bowen--Franks group $\BF(X)$ is infinite if and only if $\det(\Id-A_X) =0$.
If $\BF(X)$ is finite, then $\abs{\BF(X)} = \abs{\det(\Id-A_X)}$.
\end{lemma}

\begin{proof}
Set $M = \Id-A_X$ of rank $r$.
This is square matrix with integer coefficients and hence admits a Smith normal form; i.e., there exist invertible matrices $U,V$ with integer coefficients such that
\[
M = UDV \text{ where } D=\operatorname{diag}(d_1, \ldots, d_r, 0 \ldots, 0),
\]
satisfying $d_i \mid d_{i+1}$ for all $1\leq i < r$.
Note that $\det(U)$ and $\det(V)$ are $\pm1$.
It follows that
\[
\Z^n/M\Z^n \simeq \Z^n/D\Z^n \simeq \Z^{n-r} \oplus \bigoplus_{i=1}^r \Z/d_i\Z.
\]
This group is finite exactly when the free part vanishes (i.e., $n=r$) equivalently when $\det(M)\neq 0$.
When this happens, we observe that
\[
\abs{\Z^n / M\Z^n} = \abs{\Z^n / D\Z^n}  = \abs{\det D} = \abs{\det M}. 
\]
This completes the proof.
\end{proof}

The next result shows that if the Bowen-Franks group of the base graph $X$ is finite, then the prime divisors of the size of the Bowen-Franks group of the graph in any subsequent layer is determined by the base.  

\begin{lemma}
\label{Nakayama}
If $\abs{\BF(X)} < \infty$ then $\BF(X_{\infty}) : = \varprojlim_n\BF(X_n)$ is a finitely generated $\Lambda$-torsion module.
Furthermore, if $p\nmid \abs{\BF(X)}$, then $p\nmid \abs{\BF(X_m)}$ for all $m\geq 1$.
\end{lemma}

\begin{proof}
By definition,
\[
\Div(X_\infty) = \varprojlim_m \Div(X_m) \simeq \bigoplus_{v\in V(X)}\Lambda v \quad \text{ and } \quad \Div(X_\infty)/T \simeq \Div(X).
\]
Let $A$ be the adjacency matrix of $X$. If we consider $\Div(X_m)\otimes \Z_p$ as a free $\Z_p[\Gamma/\Gamma^{p^m}]$-module of rank $\vert V(X)\vert$.
The Adjacency matrix $A_m$ of $X_m$ can be written as $\theta A$.
In particular, there is a natural surjection $(I-\theta A)\Div(X_\infty)\to (I-A_m)\Div(X_m)\otimes \Z_p$.
Consider the diagram
\begin{align*}
\begin{matrix}
0 & \longrightarrow & (I-\theta A)\Div(X_\infty) & \longrightarrow &  \Div(X_\infty) & \longrightarrow & \BF(X_\infty) & \longrightarrow 0
\cr \hbox{ } && \lda \alpha_m && \lda \beta_m &&\lda \gamma_{m} & \hbox{} \cr 0 & \longrightarrow & (I-A_m)\Div(X_m) \otimes \Z_p & \longrightarrow & \Div(X_m) \otimes \Z_p & \longrightarrow & \BF(X_m) \otimes \Z_p & \longrightarrow 0
\end{matrix}
\end{align*}
We first note that each of the maps $\alpha_m$, $\beta_m$, and $\gamma_m$ are surjective.
By an application of the snake lemma
\[
0 \longrightarrow \ker(\alpha_m) \longrightarrow \ker(\beta_m) \longrightarrow \ker(\gamma_m) \longrightarrow 0.
\]
The map $\beta_m$ is the reduction from the infinite tower to the $m$-th finite layer.
So, $\ker(\beta_m) = \omega_m \Div(X_\infty)$ with $\omega_m = (1+T^{p^m})-1$.
When $m=0$, note that $\omega_0 = (1+T)-1 =T$.
Thus, $\ker(\beta_0) = T\Div(X_\infty)$.
Now,
\begin{align*}
    \BF(X) \otimes \Z_p & \simeq \BF(X_\infty)/\ker(\gamma_0)\\
    & \simeq \Div(X_\infty)/(T\Div(X_\infty) + (I-\theta A)\Div(X_\infty))\\
    & \simeq \BF(X_\infty)/T\BF(X_\infty).
\end{align*}

Thus, $\BF(X_\infty)/T\BF(X_\infty)$ is finite.
In particular, $\BF(X_\infty)$ is $\Lambda$-torsion. 
If $p\nmid \vert \BF(X)\vert$, by Nakayama's lemma, we see that the $p$-part of $\BF(X_\infty)$ is $0$.
Repeating the argument as above
\[
\BF(X_m)\otimes \Z_p \simeq \BF(X_\infty)/\omega_m \BF(X_\infty)
\]
and the claim follows.
\end{proof}

\begin{lemma}
\label{lemma 2.13}
Let $X$ be a finite connected digraph and $X_m$ be the $m$-th layer of the $\Z_p$-tower.
Then
\begin{align*}
a(X_m) & = a(X) + \sum_{k=1}^{p^m-1}\text{algebraic multiplicity of $\zeta_{p^m}^k$ as a root of $\det(\Id-\theta A_X)$}\\
b(X_m) & = b(X) + \sum_{k=1}^{p^m-1}\dim_{\Q(\zeta_{p^m}^k)}(\ker(\Id-\zeta_{p^m}^k A_X)).
\end{align*}    
\end{lemma}

\begin{proof}
See \cite[Lemma~7.5]{lei2026iwasawa}  
\end{proof}

\begin{definition}
\label{strongly connected}
Let $X$ be a digraph.
A path of length $k$ in $X$ is a sequence of edges $e_1, \ldots, e_k \in E(X)$ such that $t(e_i) = o(e_{i+1})$ for all $1 \leq i \leq k-1$.
The digraph $X$ is called \emph{strongly connected} if there is a path from $v$ to $v'$ in $X$ for all pairs $(v,v') \in V(X)\times V(X)$.
\end{definition}

\begin{lemma}
Let $X$ be a finite graph with $k$ strongly connected components $X_{[1]},\dots X_{[k]}$ that are pairwise isomorphic.
Then 
\[
\BF(X)=\BF(X_{[1]})^k.
\]
\end{lemma}

\begin{proof}
Let $A_1$ be the adjacency matrix of $X_{[1]}$.
Then the adjacency matrix $A$ of $X$ is a block matrix with $A_1$ as blocks on the diagonal, and all other blocks are zero.
The claim is now immediate. 
\end{proof}

\section{Classes of Examples}
\label{classesofexamples}
We compute the Bowen--Franks group for classes of (directed) Cayley graphs $X$.
We study the defect $\delta(X_m)$ for each layer of the constant $\Z_p$-extension of $X$.
This provides evidence towards a conjecture of Lei--M{\"u}ller \cite{lei2026iwasawa} which predicts that $\delta(X_m)$ is constant for all $m$.
When possible, we compute the associated Iwasawa invariants.

\subsection{Cycle Graphs}
As a warm-up exercise, we start with a simple example.
Consider the cyclic group $G=C_n$ and the associated Cayley graph $X = X(C_n,S)$ where $S=\{g\}$ is the generating set.
This is a directed cycle graph with $n$ vertices.

\begin{center}
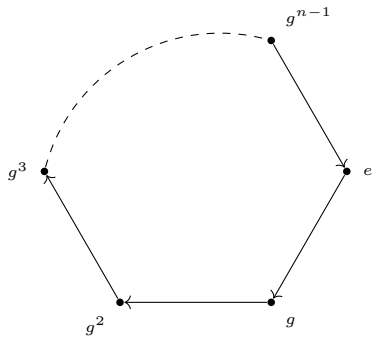
\begin{figure}[h!]
\begin{tikzpicture} [%
    nd/.style = {circle,fill=black,text=white,inner sep=1pt},
    tn/.style = {node distance=1pt},
    redarrow/.style={->, red, fill=none,>=stealth},
    blueline/.style={->,blue,fill=none}]
    
\def\r{2.0} 

\node[nd] (v0) at (0:\r) {};
\node[nd] (v1) at (60:\r) {};
\node[nd] (v3) at (180:\r) {};
\node[nd] (v4) at (240:\r) {};
\node[nd] (v5) at (300:\r) {};

\node[tn] [right=of v0] {\tiny{$e$}};
\node[tn] [above right=of v1] {\tiny{$g^{n-1}$}};
\node[tn] [left=of v3] {\tiny{$g^3$}};
\node[tn] [below left=of v4] {\tiny{$g^2$}};
\node[tn] [below right=of v5] {\tiny{$g$}};

    \draw[->] (v1) -- (v0);
    \draw[dashed] (v3) to[bend left=45] (v1);
    \draw[->] (v4) -- (v3);
    \draw[->] (v5) -- (v4);
    \draw[->] (v0) -- (v5); 
    
\end{tikzpicture}
\caption{Cayley graph of $C_{n}=  \langle g \mid g^n = e\rangle$}
\end{figure}
\end{center}

\begin{theorem}
Let $X = C_n$ be a cycle graph with $n$ vertices.
Then $\mathcal{BF}(X)$ is infinite and defect $\delta(X)=0$.
Fix a prime $p$.
Let $X_m$ denote the $m$-th layer of the constant $\Z_p$-cover of $X$. Then the defect $\delta(X_m) = 0$.
The Iwasawa invariants $\mu_p(\mathcal{BF}(X_\infty)) = 0$ and $\lambda_p(\mathcal{BF}(X_\infty)) = 1$.
\end{theorem}

\begin{proof}
The adjacent matrix for $C_n$ is 
\[
A_{X} = \begin{bmatrix}
0 & 1 & 0 & 0 & \cdots & 0 & 0 \\
0 & 0 & 1 & 0 & \cdots & 0 & 0 \\
0 & 0 & 0 & 1 & \cdots & 0 & 0 \\
0 & 0 & 0 & 0 & \cdots & 0 & 0 \\
\vdots & \vdots & \vdots & \vdots & \ddots & \vdots & \vdots \\
0 & 0 & 0 & 0 & \cdots & 0 & 1 \\
1 & 0 & 0 & 0 & \cdots & 0 & 0
\end{bmatrix}
\] 
We calculate that the characteristic polynomial $\Char_{A_X}(x) = \det( x\Id-A_X) = x^n - 1$, so $\det(\Id - A_{X}) = 0$, indicating that the Bowen--Franks group is an infinite group.
Observe that $1$ is an eigenvalue of $A_X$ with multiplicity 1.
This means that the algebraic rank $a(X) = 1$.
On the other hand, the null-rank of the matrix $\Id-A_X$ is also 1 (as it must be) which means that the geometric rank $b(X) = 1$.
Thus, the defect $\delta(X) = 0$.

By definition, 
\[
\mathcal{BF}(X) = \Div(X)/(\Id-A_X)\Div(X) \cong \Z^n/(\Id-A_X)\Z^n.
\]

\noindent \textit{Claim:}
With notation as above  $\mathcal{BF}(X) \cong \Z$.

\medskip

\noindent \textit{Justification:}
To prove the claim, it suffices to show that $\rk(\Id-A_X)$ is $n-1$ and that the space spanned by the columns of $\Id-A$ can be complemented to a $\Z$-basis of $\Z^n$.
The matrix 
\[
\Id-A_X = \begin{bmatrix}
1 & -1 & 0 & 0 & \cdots & 0 & 0 \\
0 & 1 & -1 & 0 & \cdots & 0 & 0 \\
0 & 0 & 1 & -1 & \cdots & 0 & 0 \\
0 & 0 & 0 & 1 & \cdots & 0 & 0 \\
\vdots & \vdots & \vdots & \vdots & \ddots & \vdots & \vdots \\
0 & 0 & 0 & 0 & \cdots & 1 & -1 \\
-1 & 0 & 0 & 0 & \cdots & 0 & 1
\end{bmatrix}
\]To show that $\rk(\Id-A_X) = n-1$, we notice that all rows except the $n$-th row are reduced.
Next, the sum of all the rows is $\begin{bmatrix} 0 & \cdots & 0\end{bmatrix}$, so we can row reduce $\Id-A_X$ to get the matrix 
\[
\begin{bmatrix}
1 & -1 & 0 & 0 & \cdots & 0 & 0 \\
0 & 1 & -1 & 0 & \cdots & 0 & 0 \\
\vdots & \vdots & \vdots & \vdots &  & \vdots & \vdots \\
0 & 0 & 0 & 0 & \cdots & 1 & -1 \\
0 & 0 & 0 & 0 & \cdots & 0 & 0 
\end{bmatrix}.
\]
This shows that the rank is $n-1$. If we consider the vector space spanned by the columns of $\Id-A_X$ and the first standard basis vector $e_1$, we can easily see that this span contains 
\[e_2=v_2+e_1,\]
where $v_2$ is the second column of the matrix. In a similar manner we can now generate each $e_i$ for $1\le i\le n$.
This also completes the proof of the claim.


Now we compute the Bowen--Franks group for the Galois cover $X_1$ corresponding to the first layer $\Z/p\Z$-extension. Such cover contains $pn$ vertices labelled as
\[
\begin{matrix}
    (v_1, 1), &(v_2, 1), &\cdots, &(v_n, 1)\\
    (v_1, g), &(v_2, g), &\cdots, &(v_n, g)\\
     \vdots  & \vdots &\vdots & \vdots  \\
    (v_1, g^{p-1}), &(v_2, g^{p-1}), &\cdots, &(v_n, g^{p-1})
\end{matrix}
\]
The edges of $X_1$ are given in the following manner
\[
(v_i, g^j) \mapsto (v_{(i+1 \bmod n)}, g^{(j+1 \bmod p)}).
\]
Therefore, if $p\mid n$, then $X_1$ is a disconnected graph consisting of $p$-many cyclic graphs (each) with $n$-many vertices.
If we start at the vertex $(v_1,1)$ and propagate along the edges of $X_m$, the next vertex with first component $v_1$  we reach is $(v_1,g^n)$.
If we continue this procedure, we reach $(v_1,g^{kn})$ for $1\le k\le p^{m-\val_p(n)}$.
Thus, in $X_m$ we have $p^{m-\val_p(n)}$ pre-images of $v_1$ that lie in the same connected component.
This shows there exists $m_0$ such that for all $m>m_0$, the number of connected components stops growing.
If $p\nmid n$, we have $m_0=0$.
For calculating Iwasawa invariants (from determinants of matrices) we fix one connected component at level $m\geq m_0$ and consider the $\Z_p$-tower above this cyclic graph.
The same calculations as those done for the base level imply that the defect $\delta(X_m) = 0$.
Indeed, for each connected cyclic component, this is straightforward: a directed cycle has adjacency matrix with eigenvalue $1$ of algebraic and geometric multiplicity $1$.
For the full disconnected graph, it is also still true, because if there are $c_m$ identical cyclic components, then
$a(X_m)=c_m$ and $b(X_m)=c_m$, so
$\delta(X_m)=0$.

\textit{Recall that if the graph is not connected, we restrict to a tower of connected components.
As $X_{m_0}$ consists of $p^{m_0}$ copies of the base graph this is still a $\Z_p$-tower of $X_0$.}

Repeating our calculations from above, the Bowen--Franks group for each connected component at each layer is isomorphic to $\Z$.
Recall from Lemma~\ref{Nakayama} that $\BF(X_\infty)$ is presented over $\Lambda$ by the endomorphism $\Id-\theta A_X$ of the free $\Lambda$-module $\Div(X_\infty)\simeq \Lambda^n$.
Its characteristic element is, up to a unit in $\Lambda$, given by $\det(\Id-\theta A_X)$.
Under the identification $\theta=1+T$, this becomes
$\det(\Id-(T+1)A_X)$.
To calculate the $\mu$ and $\lambda$ invariants we note that
\[
\abs{\det(\Id- (T+1) A_X)} = \abs{(T+1)^n -1} = \abs{T^n + nT^{n-1} + \ldots nT}
\]
Since the coefficient of the highest order term is 1, it follows that 
\[
\mu_p(\mathcal{BF}(X_{\infty})) = 0.
\]
When $p\nmid n$, it is immediate from the above expression of $\det(1-(T+1)A_X)$ that $\lambda_p(\BF(X_\infty))=1$. 
On the other hand, when $p\mid n$, we climb up the tower up to layer $X_{m_0}$ and consider the invariants of a single connected component $X_0$.
We said previously that $m_0 = \val_{p}(n)$.
Writing $n=p^{m_0} r$, we notice that 
\[
(1+T)^n -1 \equiv (1+T)^{rp^{m_0}} -1
\equiv (1+T^{p^{m_0}})^r - 1 \pmod{p} \]
Thus, the first coefficient not divisible by $p$ is $T^{p^{m_0}}$.
This shows that the $\lambda$-invariant of each connected component $X_0$ is precisely 1.
\end{proof}

\subsection{Dihedral Graphs}

The dihedral group of order $2n$ has presentation 
\[
D_{2n} = \langle \sigma, \tau \mid \sigma^n = \tau^2 = e, \tau \sigma \tau = \sigma^{-1}\rangle.
\]
We consider the Cayley graph $X(D_{2n},S)$ where $S=\{\sigma, \tau\}$.

\begin{center}
\begin{figure}[h!]
\begin{tikzpicture} [%
    nd/.style = {circle,fill=black,text=white,inner sep=1pt},
    tn/.style = {node distance=1pt},
    redarrow/.style={->, red, fill=none,>=stealth},
    blueline/.style={->,blue,fill=none}]

    \node[nd] (otl) at (0,0) {};
    \node[nd] (itl) [below right=of otl] {};
    \node[nd] (itr) [right=of itl] {};
    \node[nd] (otr) [above right=of itr] {};
    \node[nd] (ibl) [below=of itl] {};
    \node[nd] (obl) [below left=of ibl] {};
    \node[nd] (ibr) [right=of ibl] {};
    \node[nd] (obr) [below right=of ibr] {};

    \draw[redarrow] (otr) -- (otl);
    \draw[redarrow] (otl) -- (obl);
    \draw[redarrow] (obl) -- (obr);
    \draw[redarrow] (obr) -- (otr); 
    \draw[redarrow] (itl) -- (itr);
    \draw[redarrow] (itr) -- (ibr);
    \draw[redarrow] (ibr) -- (ibl);
    \draw[redarrow] (ibl) -- (itl); 

    \draw[blueline] (ibl) to[bend left=30] (obl);
    \draw[blueline] (obl) to[bend left=30] (ibl);
    \draw[blueline] (itl) to[bend left=30] (otl);
    \draw[blueline] (ibr) to[bend left=30] (obr);
    \draw[blueline] (itr) to[bend left=30] (otr);
    \draw[blueline] (otl) to[bend left=30] (itl);
    \draw[blueline] (obr) to[bend left=30] (ibr);
    \draw[blueline] (otr) to[bend left=30] (itr);

    \node[tn] [below right=of itl] {\tiny{$\sigma$}};
    \node[tn] [below left=of itr] {\tiny{$\sigma^2$}};
    \node[tn] [above left=of ibr] {\tiny{$\sigma^3$}};
    \node[tn] [above right=of ibl] {\tiny{$e$}};

    \node[tn] [below left=of obl] {\tiny{$\tau$}};
    \node[tn] [below=of obr] {\tiny{$\tau\sigma  =  \sigma^3\tau$}};
    \node[tn] [above=of otl] {\tiny{$\sigma \tau =  \tau\sigma^3$}};
    \node[tn] [above=of otr] {\tiny{$\sigma^2 \tau = \tau \sigma^2$}};

\end{tikzpicture}
\caption{Cayley graph of $D_{8}=  \langle \sigma, \tau \mid \sigma^4 = \tau^2 = e, \tau \sigma \tau = \sigma^{-1}\rangle$}
\end{figure}
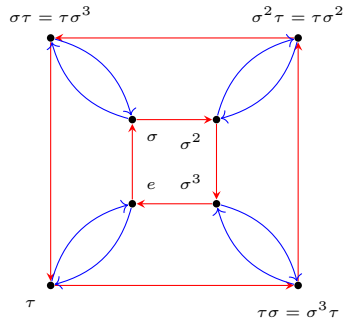
\end{center}

\begin{proposition}
\label{prop 3.2} Let $X_m$ denote the $m$-th layer of the constant $\Z_p$-cover of $X = X(D_{2n},S)$.
If $p>2$ or $n$ is odd, then $X_m$ is strongly connected.
Otherwise, $X_m$ has two isomorphic connected components.
In particular, $X_1$ consists of two copies of the base graph.
\end{proposition}

\begin{proof}
This follows from the fact that $X$ has a cycle of length $2$ and \cite[Remark~2.9]{lei2026iwasawa}.
We give a more detailed sketch for the convenience of the reader.

Let us consider the case that $p>2$.
As the base graph is strongly connected, it suffices to fix one vertex $v$ in $X = X(D_{2n},S)$ and show that $(v,1)$ is connected to every $(v,\theta^k)$ in $X_m$, where $\theta$ is a generator of $\Z/p^m\Z$.
This is equivalent to showing that there is a cycle in $X(D_{2n},S)$ through $v$ of length $k+lp^m$ for some $l$. If $p\neq 2$, every equivalence class in $\Z/p^m\Z$ contains an element divisible by $2$.
In particular, there exists a non-negative integer $k'$ such that $2k'\equiv k\pmod {p^m}$, i.e., if we loop to the cycle of length $2$ $k'$-times we obtain a path from $v$ to itself that has voltage assignment $k\pmod{p^m}$. 
The case that $n$ is odd and $p=2$ can be proved similarly using the fact that there is a cycle of odd order. 

It remains to consider the case of even $n$ and $p=2$.
In this case, every cycle has even length.
On the other hand, $2\pmod{2^m}$ generates a subgroup of order $2^{m-1}$ in $\Z/2^m\Z$.
Thus, we obtain exactly $2$ components.
As $2$ is the trivial element in $\Z/2\Z$, each connected component of $X_1$ contains exactly one vertex $w$ above any vertex $v$ of $X$. This implies that indeed $X_1$ consists of two copies of the base  graph. 
\end{proof}
\subsubsection{Representation theory}
\label{sec:represntation}

In our particular setting, $G=D_{2n}$ and $S=\{\sigma,\tau\}$ it follows from Section~\ref{cayley and rep theory} that the adjacency matrix
\[
A = \rho(\sigma) + \rho(\tau),
\] where $\rho$ denotes the regular representation of $D_{2n}$.
We recall that the irreducible representations of $D_{2n}$  depend on the parity of $n$.
When $n$ is odd, there are $\frac{n+3}{2}$ conjugacy classes, so there are exactly $\frac{n+3}{2}$ (different) irreducible representations.
There are \textit{two} 1-dimensional representations
\begin{align*}
    \text{trivial}: & \hbox{ } \sigma \mapsto 1 \text{ and } \tau \mapsto 1\\
    \operatorname{sign}_{\tau}: & \hbox{ } \sigma \mapsto 1 \text{ and } \tau \mapsto -1.
\end{align*}
The $2$-dimensional representations $\rho_j$ with $1\leq j \leq \frac{n-1}{2}$ are defined via
\[
\sigma \mapsto \begin{bmatrix}
        \zeta^j_n & 0 \\
        0 & \zeta^{-j}_n
    \end{bmatrix} \text{ and } \tau \mapsto \begin{bmatrix}
        0 & 1 \\
        1 & 0
    \end{bmatrix}.
\]
Here, we use the notation $\zeta_n = e^{\frac{2\pi j}{n}}$.
We can write down the corresponding block matrices.
For the trivial representation, $M_{\mathbf{1}} = [1+1] = [2]$ and for the alternating representation, the block $M_{\operatorname{sign}_{\tau}} = [1-1] = [0]$.
For the $2$-dimensional representations, the blocks, $M_{\rho_j} = \begin{bmatrix}
    \zeta^j_n & 1 \\
     1 & \zeta^{-j}_n
\end{bmatrix}$.

When $n$ is even, there are  $\frac{n}{2}+3$ (different) irreducible representations, \textit{four} of which are 1-dimensional
\begin{align*}
    \text{trivial}: & \hbox{ } \sigma \mapsto 1 \text{ and } \tau \mapsto 1\\
    \operatorname{sign}_{\sigma}: & \hbox{ } \sigma \mapsto -1 \text{ and } \tau \mapsto 1\\
    \operatorname{sign}_{\tau}: & \hbox{ } \sigma \mapsto 1 \text{ and } \tau \mapsto -1\\
    \operatorname{sign}_{\sigma,\tau}: & \hbox{ } \sigma \mapsto -1 \text{ and } \tau \mapsto -1.
\end{align*}
The $2$-dimensional representations $\rho_j$ with $1\leq j \leq \frac{n}{2}-1$ are defined as before.
We can write down each corresponding matrix as in the previous case.

\subsubsection{Analysis of Bowen--Franks groups of \texorpdfstring{$X$}{}}
\label{sec: size and structure}

In this section, we study the size and structure of the Bowen--Franks group of the base graph $X$.

\begin{proposition}
\label{prop 3.3}
Consider the Cayley graph $X = X(D_{2n},S)$ where $S=\{\sigma, \tau\}$.    
Then $\BF(X)$ is infinite if and only if $6\mid n$.
Otherwise $\BF(X)$ is finite and
\[
\abs{\BF(X)} = \begin{cases}
    1 & \text{ if } n \equiv \pm1 \pmod{6}\\
    3 & \text{ if } n \equiv \pm2 \pmod{6} \\
    4 & \text{ if } n \equiv 3 \pmod{6}.
\end{cases}
\]
\end{proposition}

\begin{proof}
Let $A =A_X$ be the adjacency matrix of the directed Cayley graph.
From our previous discussion, we calculate that the characteristic polynomial.
Each two-dimensional irreducible occurs with multiplicity $2$ in the regular representation and contributes a factor $x^2(x-(\zeta_n^j+\zeta_n^{-j}))^2$.
The one-dimensional representations contribute eigenvalues $2,0$ (resp. $2,-2,0,0$) when $n$ is odd (resp. even).
Hence the total zero multiplicity is $n$, giving $x^n$, while the non-zero one-dimensional eigenvalues give the factor $(x-2)$ when $n$ is odd and the factors $(x-2)(x+2)$ when $n$ is even.
Thus,
\begin{equation}
\label{char A}
\Char_A(x) = \begin{cases}
 x^n(x-2) \displaystyle\prod_{j=1}^{(n-1)/2} \left(x-(\zeta^j_n + \zeta^{-j}_n) \right)^2  & \text{ if } n \text{ is odd},\\
  x^n(x-2)(x+2) \displaystyle\prod_{j=1}^{(n/2)-1} \left(x-(\zeta^j_n + \zeta^{-j}_n) \right)^2   & \text{ if } n \text{ is even}.    
\end{cases}
\end{equation}

Recall from Lemma~\ref{BFX infinite} that $\BF(X)$ is infinite precisely when $\det(\Id-A) =0$.
This is determined by when 1 is an eigenvalue of $A$.
On the other hand, when $\abs{\BF(X)}$ is finite, we know
\[
\abs{\BF(X)} = \abs{\det(\Id-A)} = \prod_i \vert 1-\lambda_i\vert,
\]
where $\lambda_i$ are the eigenvalues of $A$.
In other words, it will suffice to calculate 
\[
\Char_A(1) = \begin{cases}
 (-1) \displaystyle\prod_{j=1}^{(n-1)/2} \left(1-(\zeta^j_n + \zeta^{-j}_n) \right)^2  & \text{ if } n \text{ is odd},\\
  (-3) \displaystyle\prod_{j=1}^{(n/2)-1} \left(1-(\zeta^j_n + \zeta^{-j}_n) \right)^2   & \text{ if } n \text{ is even}.    
\end{cases}
\]
We need to study the part involving the product in the above formula.

Suppose that $n$ is odd.
We know
\begin{align*}
    \displaystyle\prod_{j=1}^{(n-1)/2} \left(1-(\zeta^j_n + \zeta^{-j}_n) \right)^2 & = \displaystyle\prod_{j=1}^{(n-1)/2} \left(1-(\zeta^j_n + \zeta^{n-j}_n) \right)^2\\
    & = \displaystyle\prod_{j=1}^{n-1} \left(1-(\zeta^j_n + \zeta^{-j}_n) \right) \\
    & = \displaystyle\prod_{j=1}^{n-1} \left(\zeta^{2j}_n - \zeta^{j}_n + 1 \right)\\
    & = \displaystyle\prod_{j=1}^{n-1}\left( \zeta^j_n - \zeta_6\right)\left( \zeta^j_n - \zeta^{-1}_6\right) \\
    & = \left( \frac{\zeta^n_6 -1}{\zeta_6 - 1}\right)\left( \frac{\zeta^{-n}_6 - 1}{\zeta^{-1}_6 - 1}\right)\\ 
    & = \frac{2- (\zeta^n_6 + \zeta^{-n}_6)}{2-(\zeta_6 + \zeta^{-1}_6)}\\
    & = 2- 2\cos(2\pi n/6) =\begin{cases}
        1 & \text{ if } n\equiv \pm 1\pmod{6} \\
        4 & \text{ if } n\equiv 3\pmod{6}.
    \end{cases}
\end{align*}
Suppose that $n$ is even.
We observe that
\[
\frac{z^n-1}{z^2 -1} = \prod_{\substack{i=1 \\ i\neq n/2}}^{n-1} (z-\zeta^i_n) = \prod_{i=1}^{(n/2)-1} (z^2 -2\Re(\zeta^i_n)z + 1).
\]
But $\zeta^2_6 + 1 = \zeta_6$, so if we choose $z=\zeta_6$ above, then
\[
\frac{\zeta^n_6 - 1}{\zeta^2_6 -1} = \zeta^{(n/2)-1}_6 \prod_{i=1}^{(n/2)-1}  (1-2\Re(\zeta^i_n))
\]
When $n\equiv 0\pmod{6}$ it is straightforward to note that the left side is $0$.
This corresponds to the case that the Bowen--Franks group is infinite.
Furthermore, depending on whether $n\equiv \pm2\pmod{6}$, the left side of the equality is $1$ or $\zeta_6$.
Then
\[
\prod_{i=1}^{(n/2)-1}  (1-2\Re(\zeta^i_n)) =
\begin{cases}
 \frac{1}{\zeta^{(n/2)-1}_6} = 1 & \text{ if } n\equiv 2\pmod{6}\\
 \frac{\zeta_6}{\zeta^{(n/2)-1}_6} = 1 & \text{ if } n\equiv 4\pmod{6}.
\end{cases}
\]
So, in either case
\[
\prod_{i=1}^{(n/2)-1}  (1-2\Re(\zeta^i_n))^2 =1.
\]
The result is now immediate.
\end{proof}

Once we know the size of $\BF(X)$ from Proposition~\ref{prop 3.3}, the precise structure of $\BF(X)$ is immediate when $n\equiv \pm 1, \pm 2 \pmod{6}$.
In the next result, we determine the group structure when $n\equiv 3\pmod{6}$.

\begin{proposition}
\label{prop 3.4}
Consider the Cayley graph $X = X(D_{2n},S)$ where $S=\{\sigma, \tau\}$.    
Suppose that $6\nmid n$.
Then
\[
\BF(X) = \begin{cases}
    \{1\} & \text{ if } n \equiv \pm1 \pmod{6}\\
    \Z/3\Z & \text{ if } n \equiv \pm2 \pmod{6} \\
    \Z/2\Z \times \Z/2\Z & \text{ if } n \equiv 3 \pmod{6}.
\end{cases}
\]
\end{proposition}


\begin{proof}
We focus on the proof of the non-trivial case.
Using the block structure we have used before, we see that $A$ has an eigenvalue $-1$ with multiplicity $2$; this corresponds to the eigenvalue coming from the index $j=n/3$.
This means that $2$ is an eigenvalue of $\Id -A$ with algebraic and geometric multiplicity $2$.
Let $W\subset \Div(X)\otimes \Q$ be the $2$-dimensional eigenspace for $\Id-A$ corresponding to the eigenvalue $2$; i.e.,
\[
(\Id-A)w = 2w \text{ for } w\in W.
\]
Let $L'=\textup{Div}(X)\cap W$.
This is the integral lattice sitting inside that rational eigenspace.

As $W$ is a $\Q$-vector space, there is a sublattice $L''$ such that 
\[
\Div(X)=L'\oplus L''.
\]
With respect to a $\Z$-basis that respects the above decomposition, we note that since $L'\subseteq W$ and $\Id-A$ acts as multiplication by $2$ on $W$, the restriction to $L'$ is
\[
(\Id-A)\vert_{L'} = 2\Id_{2}.
\]
Therefore, the matrix $\Id-A$ has the following structure
\[
\begin{pmatrix}
        A' & B\\
        0 & D
    \end{pmatrix},
\]
where $A'=2\Id_{2}$, $B\in \textup{Mat}_{2\times (2n-2)}(\Z)$, and $D\in \GL_{(2n-2)}(\Z)$. 
By Proposition~\ref{prop 3.3}, we observe that
\[
\abs{\det(\Id-A)} = \abs{\BF(X)} = \det(2\Id_2)\abs{\det(D)} =4\abs{\det(D)} = 4.
\]
So, $D$ is a unimodular matrix in $\GL_{(2n-2)}(\Z)$.
It follows that the last $2n-2$ columns of this matrix form a partial basis of $\Div(X)$.
Since $D^{-1}$ has integer entries, integral row operations can remove the $B$-block.
This goes on to show that we can perform row and column operations to obtain
\[
\Id-A \sim \begin{pmatrix}
        2\Id_2 & 0\\
        0 & D
    \end{pmatrix}.
\]
Since $D$ is unimodular, we note that $\coker(D) = \Z^{2n-2}/D\Z^{2n-2}=0$. Hence
\[
\coker(\Id-A) \simeq \coker(2\Id_2) 
\]
It follows that 
\[
\BF(X) = \coker(\Id-A) \cong L'/2L'=\Z/2\Z\times \Z/2\Z. \qedhere
\]
\end{proof}


\subsubsection{Iwasawa Invariants of Bowen--Franks groups of \texorpdfstring{$X_\infty$}{} when \texorpdfstring{$\BF(X)$}{} is finite}

Recall from our earlier discussion that $\BF(X_\infty) = \coker(\Id-\theta A)$ acting on the free $\Lambda$-module $\Div(X_\infty)$.
Its characteristic element is, up to a unit, $\det(\Id-\theta A)$.
Under the identification $\Lambda\simeq \Z_p\llbracket T\rrbracket$ given by $\theta\mapsto 1+T$, this becomes $\det(\Id-(T+1)A)$.
This will allow us to use the definition of the Iwasawa invariants and compute them explicitly.
Since we have already calculated $\Char_A(x) = \det(x\Id-A)$ in \eqref{char A}, it is now easy to see that
\begin{equation}
\label{det(I - (T+1)A)}
\det(\Id - (T+1)A) = \begin{cases}
 1^n(1-2(T+1)) \displaystyle\prod_{j=1}^{(n-1)/2} \left(1-2\Re(\zeta^j_n)(T+1) \right)^2  & \text{ if } n \text{ is odd},\\
  1^n(1-2(T+1))(1+2(T+1)) \displaystyle\prod_{j=1}^{(n/2)-1} \left(1-2\Re(\zeta^j_n)(T+1) \right)^2   & \text{ if } n \text{ is even}.    
\end{cases}
\end{equation}
When $n$ is odd, we have seen in the previous section that $\abs{\BF(X)}$ is finite.
Moreover it is non-trivial precisely when $3\mid n$.
In this case, by Lemma~\ref{Nakayama}, the only prime of interest is $p=2$. 
We note
\begin{align*}
\det(\Id - (T+1)A) & = (-2T - 1) \displaystyle\prod_{j=1}^{(n-1)/2} \left(2\Re(\zeta^j_n)T + (2\Re(\zeta^j_n)-1)\right)^2\\
& = (-2T - 1)(-T-2)^2 \prod_{\substack{j=1 \\ j\neq n/3}}^{(n-1)/2} \left(2\Re(\zeta^j_n)T + (2\Re(\zeta^j_n)-1)\right)^2.
\end{align*}
From our computation in the previous section, each term $(2\Re(\zeta^j_n)-1)$ is in fact a unit.
This means that the contribution of the terms appearing inside the product is trivial to the Iwasawa invariant.
Thus, upon opening the brackets and carrying out the multiplication we observe that the associated $\mu$-invariant is 0 and $\lambda$-invariant is 2.
When $p\neq 2$, both the Iwasawa invariants are 0.

Next we treat the case that $n$ is even.
We perform calculations like in the previous case and see that
\[
\det(\Id - (T+1)A) = (-2T - 1)(2T+3) \displaystyle\prod_{j=1}^{(n/2)-1} \left(2\Re(\zeta^j_n)T + (2\Re(\zeta^j_n)-1)\right)^2.
\]
If $3\nmid n$, the Bowen--Franks group $\abs{\BF(X)}$ is finite, and by Lemma~\ref{Nakayama} the only prime of interest is $p=3$.
As before each term $(2\Re(\zeta^j_n)-1)$ is a unit.
This results in $\mu=0$.
The smallest power of $T$ with coefficients not divisible by 3 is $1$ and hence $\lambda_3=1$.
Since $\mu=0$, the $\lambda$-invariant by definition is equal to the order of vanishing of the characteristic element modulo $p$.
When $p=3$, we see that the constant term is divisible by $3$, while the coefficient of $T$ is a $3$-adic unit.
Hence upon reducing modulo $3$, the expression is divisible by $T$ but not by $T^2$; therefore $\lambda_3=1$.
When $p\neq 3$, the Iwasawa invariants are trivial.

We summarize the calculations in the following proposition.

\begin{proposition}
\label{prop 3.6}
Consider the Cayley graph $X = X(D_{2n},S)$ where $S=\{\sigma, \tau\}$.
Suppose that $n$ is odd.
If $p\neq 2$, then $\mu_p(\BF(X_\infty))=\lambda_p(\BF(X_\infty))=0$.
If $p=2$, then $\mu_p(\BF(X_\infty))=0$ and 
\[
\lambda_p(\BF(X_\infty)) = \begin{cases}
    0 & \text{ if } 3\nmid n\\
    2 & \text{ if } 3\mid n.
\end{cases}
\]
Suppose that $n$ is even and $3\nmid n$.
If $p\neq 3$, then $\mu_p(\BF(X_\infty))=\lambda_p(\BF(X_\infty))=0$.
If $p=3$, then $\mu_p(\BF(X_\infty))=0$ and $\lambda_p(\BF(X_\infty)) =1$.
\end{proposition}

\subsubsection{Iwasawa Invariants of Bowen--Franks groups of \texorpdfstring{$X_\infty$}{} when \texorpdfstring{$\BF(X)$}{} is infinite}
In this case,
\begin{equation}
\label{det polynomial}
\det(\Id - (T+1)A) = (-2T - 1)(2T+3)T^2 \displaystyle\prod_{\substack{j=1 \\ j\neq n/6}}^{(n/2)-1} \left(2\Re(\zeta^j_n)T + (2\Re(\zeta^j_n)-1)\right)^2.
\end{equation}
Recall that to study the Iwasawa invariants of $\BF(X_{\infty})$, it suffices to study the invariants of each factor.
Rewrite
\begin{align*}
P_n(T) := \displaystyle\prod_{\substack{j=1 \\ j\neq n/6}}^{(n/2)-1} \left(2\Re(\zeta^j_n)T + (2\Re(\zeta^j_n)-1)\right)^2 & = \displaystyle\prod_{\substack{k\mid n \\ k\neq 2,6}} \prod_{\substack{j=1 \\ \operatorname{ord}(\zeta^j_n)=k}}^{(n/2)-1} \left(2\Re(\zeta^j_n)T + (2\Re(\zeta^j_n)-1)\right)^2 \\
& =: \displaystyle\prod_{\substack{k\mid n \\ k\neq 2,6}} 
P^{(k)}_n(T). 
\end{align*}

\begin{definition}
Given two functions $f(x)=\prod_{i=1}^s (x-\alpha_i)$ and $g(x) = \prod_{i=j}^t (x-\beta_j)$, the resultant 
\[
\operatorname{Res}(f,g) := \prod_{i} g(\alpha_i) = \prod_{j} (-1)^{st} f(\beta_j).
\]
\end{definition}

\begin{lemma}
\label{resultant}
Fix $k\mid n$ and $k\neq 2,6$.
Then
\[
P^{(k)}_n(0) = \operatorname{Res}(\Phi_k, \Phi_6),
\]
where $\Phi_k(x)$ is the $k$-th cyclotomic polynomial.
\end{lemma}

\begin{proof}
We start with the left hand side and observe that
\begin{align*}
   P^{(k)}_n(0) & = \prod_{\substack{j=1 \\ \ord(\zeta^j_n)=k}}^{(n/2)-1} (\zeta^j_n + \zeta^{-j}_n -1)^2\\
   & = \prod_{\substack{j=1 \\ \ord(\zeta^j_n)=k}}^{(n/2)-1} \abs{\zeta^{-j}_n}^2 \abs{\zeta^{2j}_n - \zeta^{j}_n + 1}^2\\
   & = \prod_{\substack{j=1 \\ \ord(\zeta^j_n)=k}}^{(n/2)-1} \abs{\zeta^{2j}_n - \zeta^{j}_n + 1}^2\\
   & = \prod_{\substack{j=1 \\ \ord(\zeta^j_n)=k}}^{(n/2)-1} \Phi_6(\zeta^j_n) \Phi_6(\zeta^{n-j}_n) = \prod_{\substack{j=1 \\ \ord(\zeta^j_n)=k}}^{n-1} \Phi_6(\zeta^j_n) = \operatorname{Res}(\Phi_k, \Phi_6).
\end{align*}
Here, we used the fact that
\[
\Phi_6(\zeta^j_n) =\zeta^{2j}_n-\zeta^{j}_n+1.
\]
This completes the proof of the lemma.
\end{proof}

The coefficient of $T^2$ in \eqref{det polynomial} is $3P_n(0)$.
In what follows, we first show that a prime $q \mid 3P_n(0)$ if and only if $q\mid n$.
Therefore, we show that in the case that $6\mid n$, the only primes we have to consider for calculating Iwasawa invariants are the primes dividing $n$.

\begin{lemma}
\label{lemma to cite in 3.11}
The prime factors of $3P_n(0)$ 
are the prime divisors of $n$.
In particular, the prime divisors of the coefficient of $T^2$ in \eqref{det polynomial} are the prime divisors of $n$.
\end{lemma}

\begin{proof}
Note that we are in the situation $3\mid n$ and the coefficient of $T^2$ has a factor of 3.
So we need to prove that the prime divisors of $P_n(0)$ are precisely the prime divisors of $n$ (except possibly 3). 
By Lemma~\ref{resultant}
\[
P_n(0) 
= \prod_{\substack{k\mid n\\ k\neq 2,6}} P_n^{(k)}(0) = \prod_{\substack{k\mid n \\ k\neq 2,6}}\operatorname{Res}(\Phi_k, \Phi_6).
\]
By \cite[Theorem~4]{apostol1970resultants}, we deduce that if $q \neq 3$ is prime then $q \mid P_n$ if and only if $q\mid n$.
The condition $k>6$ required in \emph{op.~cit.} is not a serious one since $\operatorname{Res}(\Phi_k, \Phi_6) = \operatorname{Res}(\Phi_6, \Phi_k)$; see \cite[p.~459]{apostol1970resultants}. 
\end{proof}

Therefore, the primes of interest for calculating the Iwasawa invariants of $\BF(X_\infty)$ when $6\mid n$ are all the prime divisors of $n$.
Fix $k\mid n$ such that $k\neq 2,6$. 
By \cite[Theorem~4]{apostol1970resultants}, it is easy to see that for $k\ge 3$
\[
P^{(k)}_n(0) = \begin{cases}
    1 & \text{ if } k\not\in\{3, 6p^r\},\\
    4 & \text{ if } k=3,\\
    p^2 & \text{ if } k=6p^r.
\end{cases}
\]
In particular, when $k\not\in\{3, 6p^r\}$, we note that $P^{(k)}_n(T)$ is a unit and their contribution to the $\mu$ and $\lambda$ invariant calculation will be trivial for all primes $p$.
We ignore these factors for the rest of the calculation.

When $k=3$, we calculate
\[
P^{(3)}_n(T) = (2\Re(\zeta_3)T + 2\Re(\zeta_3)-1)^2 = (-T-2)^2.
\]

\begin{lemma}
\label{unit}
Fix $k=6p^r$ and $1\leq j \leq (n/2)-1$ such that $\ord(\zeta^j_n)=k$.
Then $2\Re(\zeta^j_n)$ is a unit in $\mathcal{O}_L$ where $L=\mathbb{Q}(\zeta_n)$.
\end{lemma}

\begin{proof}
We start by rewriting
\[
2\Re(\zeta^j_n) = \zeta^{-j}_n(1+\zeta^{2j}_n).
\]
Since $\ord(\zeta^j_n)=k$, it follows that $\ord(\zeta^{2j}_n)=3p^r$.
Set $f_{1+\zeta^{2j}_n}(x)$ to be the minimal polynomial of $1+\zeta_n^{2j}$.
To prove our claim, we need to show that the constant term of this minimal polynomial is a unit.
Moreover, $f_{1+\zeta^{2j}_n}(x) = \Phi_{3p^r}(x-1)$ so we are interested in $\Phi_{3p^r}(-1)$.

If $p\neq 3$ we know
\[
\Phi_{3p^r}(x) = \frac{x^{3p^r} -1}{\frac{(x^{p^r}-1)(x^{3p^{r-1}}-1)}{(x^{p^{r-1}}-1)}}.
\]
Moreover, when $p=2$, the above expression simplifies and we get
\[
\Phi_{3\cdot 2^r}(x) = \frac{x^{3\cdot 2^{r-1}} +1}{x^{2^{r-1}} +1}.
\]

If $p=3$ we know that 
\[\Phi_{3^{r+1}}=x^{2\cdot3^r}+x^{3^r}+1.\]We check that $\abs{\Phi_{3p^r}(-1)}=1$.
This completes the proof.
\end{proof}

Before stating the next result, we need to introduce modified Iwasawa invariants in a slightly extended coefficient ring.
Let $L$ be a number field, let $\mathfrak p\mid p$ be a prime of $L$ above the rational prime $p$.
Let
\[
f(T)=\sum_{i\ge 0} a_iT^i \in \mathcal O_{L,\mathfrak p} \llbracket T \rrbracket
\]
be nonzero.
Then $\mu_{\mathfrak{p}}(f) = \min_i \val_{\mathfrak{p}}(a_i)$ where $\val_{\mathfrak{p}}$ is normalized by $\val_{\mathfrak{p}}(\mathfrak{p})=1/e(\mathfrak{p})$, where $e$ is the ramification index of $\mathfrak{p}$ in $\Q_p(\zeta_n)$.
If $\mu_{\mathfrak{p}}(f)=0$, then define $\lambda_{\mathfrak{p}}(f) = \min\{i \ : \ a_i \not\in \mathfrak{p}\}$.
When $f(T)\in \Z_p \llbracket T\rrbracket$, these definitions agree with the usual Iwasawa invariants, since $\mathfrak{p}\cap\Z_p=p\Z_p$.

\begin{lemma}
\label{mu and lambda for factor}
Fix $k=6p^r$ and $1\leq j \leq (n/2)-1$ such that $\ord(\zeta^j_n)=k$.
For every $\mathfrak{p}\mid p$ of $L=\Q(\zeta_n)$,
\begin{align*}
    \mu_{\mathfrak{p}}(2\Re(\zeta^j_n)T + 2\Re(\zeta^j_n)-1) & = 0\\
    \lambda_{\mathfrak{p}}(2\Re(\zeta^j_n)T + 2\Re(\zeta^j_n)-1) & =1. 
\end{align*}
\end{lemma}

\begin{proof}
For ease of notation set $f_j(T)=2\Re(\zeta^j_n)T + 2\Re(\zeta^j_n)-1$.
Note that $2\Re(\zeta^j_n) = \zeta^j_n + \zeta^{-j}_n$ is a unit in $\mathcal{O}_{L,\mathfrak{p}}$.
Since the coefficient of $T$ in $f_j(T)$ is a $\mathfrak{p}$-adic unit, the $\mu_{\mathfrak{p}}$-invariant must be 0.

The residue field module $\mathfrak{p}$ has characteristic $p$ and its multiplicative group has order prime-to-$p$.
Reducing $\zeta_n^j$ modulo $\mathfrak{p}$ removes the $p$-power part of its order, i.e., $\overline{\zeta_n^j}$ has order equal to prime-to-$p$ part of $6p^r$.

If $p\neq 2,3$ then $\overline{\zeta_n^j}$ has order $6$ and is a primitive $6$-th root of unity.
Thus, $\zeta_n^j + \zeta_n^{-j} - 1 \in \mathfrak{p}$.
When $p=2$, then $\overline{\zeta_n^j}$ has order $3$.
Since the characteristic is 2, we know that $+1 = -1$ and once again $\zeta_n^j + \zeta_n^{-j} - 1 \in \mathfrak{p}$.
Finally, when $p=3$, we note $\overline{\zeta_n^j} = -1$ has order 2.
This again means $\zeta_n^j + \zeta_n^{-j} - 1 \in \mathfrak{p}$.
This shows that
\[
f_j(T) \equiv (\zeta_n^j + \zeta_n^{-j})T \pmod{\mathfrak{p}},
\]
where $\zeta_n^j + \zeta_n^{-j}$ is a unit.
Hence the reduction of $f_j(T)$ has a zero of order exactly $1$ at $T=0$.
This completes the proof of the lemma.
\end{proof}

The above lemma takes into account a single linear factor.
If we are interested in the Iwasawa invariant of the product $P_n^{(k)}(T)$, then we need to further note that $f_j$'s are Galois conjugates.
In other words, the product $P_n^{(k)}(T)$ is Galois-stable and must have integral coefficients.
Therefore
\[
\lambda_p(P_n^{(k)}(T)) = \sum_{j} \lambda_{\mathfrak{p}}(f_j(T)^2) = 2 \times \#\{ j \ : \ 1\leq j \leq (n/2)-1, \ \ord(\zeta^j_n)=k\} = 2 \times \frac{\varphi(k)}{2} = \varphi(k).
\]

\begin{proposition}
\label{prop 3.12}
Consider the Cayley graph $X = X(D_{2n},S)$ where $S=\{\sigma, \tau\}$.
Suppose that $6\mid n$.
The invariant $\mu_p(\BF(X_\infty))=0$ for all primes $p$ and
\[
\lambda_p(\BF(X_\infty)) = \begin{cases}
    0 & \text{ if } p\nmid n\\
    2^{\val_2(n)} & \text{ if } p=2\\
    3^{\val_3(n)} & \text{ if } p=3\\
    2p^{\val_p(n)} & \text{ if } p\mid n \text{ and } \gcd(p,6)=1.    
\end{cases}
\]
\end{proposition}

\begin{proof}
To calculate the Iwasawa invariants of the Bowen--Franks group we look at the $\mu$ and $\lambda$ invariants of each linear factor appearing on the right side of \eqref{det polynomial}.
A linear factor which is a unit has a trivial contribution to both the invariants.
We have seen previously that each for each prime $p$, the linear factors all have $\mu$-invariant 0.
This shows $\mu_p(\BF(X_\infty))=0$.

To compute the $\lambda$-invariant at the prime $p=2$, we consider
\begin{align*}
U\det(\Id - (T+1)A)  & \equiv T^2 P^{(3)}_n(T) \prod_{\substack{k\mid n\\ k=6\cdot 2^r}} P^{(k)}_{n}(T) \pmod{2}\\
& \equiv T^4 \prod_{\substack{k\mid n\\ k=6\cdot 2^r}} P^{(k)}_{n}(T) \pmod{2}
\end{align*} for some unit $U$. 
By Lemma~\ref{mu and lambda for factor}, we have $\lambda_2 ((2\Re(\zeta^j_n)T + 2\Re(\zeta^j_n)-1)^2)=2$.
The number of $j$ for which $\ord(\zeta^j_n) = k = 6\cdot 2^r$ is $\varphi(k)/2$.
Furthermore, we have to divide by $2$ to take into account the number of connected components; c.f., Proposition~\ref{prop 3.2}.
Putting this all together, we get
\[
\lambda_2(\BF(X_\infty)) =\frac{1}{2}\left( 4 + \displaystyle\sum_{r=1}^{\operatorname{val}_2(n)-1}\varphi(6\cdot 2^r)\right).
\]
The calculations for $p=3$ and $p\mid n$ but $\gcd(p,6)=1$ are similar.
In particular, we obtain
\[
\lambda_p(\BF(X_\infty)) = \begin{cases}
    0 & \text{ if } p\nmid n\\
    2+\frac{1}{2} \displaystyle\sum_{r=1}^{\operatorname{val}_2(n)-1} \varphi(6\cdot 2^r) & \text{ if } p=2\\
    3+ \displaystyle\sum_{r=1}^{\operatorname{val}_3(n)-1}\varphi(6\cdot 3^r) & \text{ if } p=3\\
    2+ \displaystyle\sum_{r=1}^{\operatorname{val}_p(n)}\varphi(6 p^r) & \text{ if } p\mid n \text{ and } \gcd(p,6)=1.    
\end{cases}
\]
The result follows from simplifying this expression.
When $p\nmid n$, the claim follows from Lemma~\ref{lemma to cite in 3.11}.
\end{proof}

\subsubsection{Defect}
We show that the conjecture of Lei--M{\"u}ller is true for Cayley graphs of dihedral groups.

\begin{proposition}
\label{defect base}
Consider the Cayley graph $X = X(D_{2n},S)$ where $S=\{\sigma, \tau\}$.
Then the defect $\delta(X)=0$.
\end{proposition}

\begin{proof}
Recall that $1$ is an eigenvalue of the adjacency matrix $A=A_X$ if and only if $6\mid n$.
So, when $6\nmid n$, it is trivially true that $\delta(X)=0$.
On the other hand, when $6\mid n$, we check from $\Char_A(x)$ calculated in \eqref{char A} that the eigenvalue 1 has multiplicity $2$, i.e.~the algebraic rank $a(X)=2$.
From the representation theory description of $A$, 
we know that each contribution of eigenvalue 1 in this case comes from a 2-dimensional representation with the corresponding block matrix
\[
M_{\rho_{n/6}} = \begin{bmatrix}
    \zeta^{n/6}_n & 1\\
    1 & \zeta^{5n/6}_n
\end{bmatrix}
\]
which has characteristic polynomial $x(x-2\Re(\zeta^{n/6}_n))$.
But there are two such blocks in the decomposition of $A$ (since this is a 2-dimensional representation) -- the eigenvectors corresponding to eigenvalue 1 are thus linearly independent.
This shows that the geometric rank $b(X)=2$.
The defect $\delta(X) = a(X)-b(X)=0$.
\end{proof}

\begin{theorem}
\label{thm: defect conj for dihedral}
Consider the Cayley graph $X = X(D_{2n},S)$ where $S=\{\sigma, \tau\}$.
For each layer $X_m$ of the constant $\Z_p$-extension of $X$, the defect $\delta(X_m)=0$ for all $m\geq 0$.
\end{theorem}

\begin{proof}
The case $m=0$ is addressed in the previous result.
Fix $m>0$.
Observe that $a(X_m) \geq a(X)$.
If equality holds, there is nothing to show.
Otherwise, there exists a (non-trivial) $p$-power root of unity which appears as an eigenvalue of $A$.
In particular,
\[
\det(\Id-\zeta^k_{p^m}A)=0 \text{ for some }k\leq p^m.
\]
We know from \eqref{det(I - (T+1)A)} that $\det(\Id-\theta A)$ is a product of linear factors with real coefficients; so the only $p$-power roots of unity that can appear as eigenvalues of $A$ are $\pm 1$.
This forces $p=2$.

The case of $\theta= 1$ corresponds to the base layer which we have considered in Proposition~\ref{defect base}.
If we require that $\det(\Id+A)=0$, this happens precisely when $\zeta^j_n = \zeta_3$, 
i.e.~if and only if $3\mid n$ and $j=n/3$.
Note that we showed in Lemma~\ref{lemma 2.13} that
\begin{align*}
a(X_m) & = a(X) + \sum_{k=1}^{p^m-1}\text{algebraic multiplicity of $\zeta_{p^m}^k$ as a root of $\det(\Id-\theta A)$}\\
b(X_m) & = b(X) + \sum_{k=1}^{p^m-1}\dim_{\Q(\zeta_{p^m}^k)}(\ker(\Id-\zeta_{p^m}^k A)).
\end{align*}    
But 
\[
\dim(\ker(\Id -\zeta_{p^m}^k A)) = \text{algebraic multiplicity of $\zeta_{p^m}^k$ as a root of $\det(\Id-\theta A)$}
\]
by the block matrix argument in Proposition~\ref{defect base}.
\end{proof}

\subsection{Cayley graphs of \texorpdfstring{$\Z/p\Z\rtimes \Z/(p-1)\Z$}{}}
\label{sec: semi direct}

Fix an odd prime $p$ and {a primitive root $a$ modulo $p$}.
Consider the group of order $p(p-1)$ with presentation 
\[
G = \langle \sigma, \tau \mid \sigma^p = \tau^{p-1} = e, \tau \sigma \tau^{-1} = \sigma^{a}\rangle {\simeq \Z/p\Z\rtimes \Z/(p-1)\Z}.
\]
We consider the Cayley graph $X=X(G,S)$ where $S=\{\sigma, \tau\}$.

The strategy is similar to the previous class of examples and we use \eqref{adjacency} to construct the adjacency matrix.
There are $(p-1)$-many $1$-dimensional representations and each such representation contributes a $1\times 1$ block to the adjacency matrix $A=A_X$.
On the other hand, {there is a unique} $(p-1)$-dimensional representation ${\rho}$.
{The dimension check is
\[
(p-1)\cdot 1^2 + 1 \cdot (p-1)^2 = p(p-1).
\]}
{Since an irreducible representation occurs in the regular representation with multiplicity equal to its dimension, $\rho$ occurs with multiplicity $p-1$.
Hence $A$ contains $p-1$ copies of the block $M_{\rho}$ given by
\[
M_{\rho} = \begin{bmatrix}
    \zeta_p & 1 & 0 & \ldots & 0 \\
     0  &  \zeta^{a}_p & 1 & \ldots & 0 \\
     \vdots & \ddots & \ddots & \ddots & \vdots\\
     1 & 0 & 0 & \ldots & \zeta^{a^{p-2}}_p
\end{bmatrix}.
\]}
We calculate the characteristic polynomial
\[
\Char_{M_{\rho}}(x) = \det( x\Id-M_{\rho}) = \Phi_p(x) - 1 = x + x^2 + \ldots + x^{p-1}.
\]
Here, $\Phi_p(x)$ is the $p$-th cyclotomic polynomial.
The roots of the characteristic polynomial are easily seen to be $0$ or $\zeta^k_{p-1}$ where $1 \leq k \leq p-2$.
{In particular, all roots are distinct, so $M_\rho$ is diagonalizable.
Now, each one-dimensional block is of course diagonalizable, and $M_\rho$ is diagonalizable; so, the entire adjacency matrix $A$ is diagonalizable.
Consequently, every eigenvalue of $A$ has equal algebraic and geometric multiplicity.}
We conclude that $\delta(X)=0$.
Writing $X_m$ to denote the Cayley graph corresponding to the $m$-th layer of the $\Z_p$-extension of $X$ and {using Lemma~\ref{lemma 2.13}} we know that {$a(X_m)$ and $b(X_m)$ must agree term-by-term.
Thus,} $\delta(X_m)=0$ for all $m\geq 0$.

Next we want to compute the Iwasawa invariants.
The strategy is the same as the one we have used in all the previous examples.
We have
\[
\det(\Id - \theta A) = \prod_{k=1}^{p-2} (1 - \theta \zeta^k_{p-1})^{p-1} \prod_{k=1}^{p-1} (1 - (1+\zeta^k_{p-1})\theta).
\]
Now replacing $\theta = T+1$, we obtain
\begin{align*}
\det(\Id - (T+1) A) &= \prod_{k=1}^{p-2} (1 - (T+1) \zeta^k_{p-1})^{p-1}\prod_{k=1}^{p-1} (1 - (1+\zeta^k_{p-1})(T+1))\\ 
&= \prod_{k=1}^{p-2} (1 - (T+1) \zeta^k_{p-1})^{p-1} \prod_{k=1}^{p-1} (-(1+\zeta^k_{p-1})T - \zeta^k_{p-1}) .
\end{align*}
Since the constant term arises in the second product is $\zeta^k_{p-1}$ and we note that it is always a unit, this means its contribution to the Iwasawa invariant is 0 for every prime $q$.
We only need to focus on the first product.

Observe that
\[
1 - (1+T) \zeta^k_{p-1} = (1 - \zeta^k_{p-1}) - \zeta^k_{p-1}T.
\]
The coefficient of $T$ is always a unit.
This gives $\mu_q=0$ for every prime $q$.

{The factor has positive $\lambda_q$-invariant  precisely when $1-\zeta^k_{p-1}$ is a non-unit at a prime above $q$.}
The element $1-\zeta_{p-1}^k$ is a unit in the ring of integers of $\mathbb{Q}(\zeta_{p-1})$ if and only if the integer $d= \frac{p-1}{\gcd (p-1,k)}$ has \textit{at least} two distinct prime factors.
So the only primes with a positive $\lambda$-invariant are the primes $q$ for which there are non-negative integers $k$ and $s$ satisfying 
\[
\frac{p-1}{\gcd (p-1,k)} = q^s.
\]
This is the case if and only if $q$ is a divisor of $p-1$. For each such choice of $k$ we obtain
\[\lambda_q(\pm \zeta_{p-1}^kT+\zeta_{p-1}^k-1)=1\]
And these are the only terms that contribute to $\lambda_q$. 
For each $q^s\mid p-1$, there are $\varphi(q^s)$ different choices of $k$. For $q\mid p-1$, we therefore obtain 
\[
\lambda_q = (p-1)\sum_{t=1}^{\val_q(p-1)} \varphi(q^t)=(p-1)(q^{\val_q(p-1)}-1).
\]

{The above discussion can be summarized as follows.}

\begin{proposition}
Let $p>2$ be a prime and let $a$ be a primitive root $p$.
Consider the group
\[
G = \langle \sigma, \tau \mid \sigma^p = \tau^{p-1} = e, \tau \sigma \tau^{-1} = \sigma^a\rangle \simeq \Z/p\Z\rtimes \Z/(p-1)\Z.
\]
Let $X=X(G,\{\sigma, \tau\})$ be the associated Cayley graph.
Then $\delta(X_m)=0$ for every layer $X_m$ of the constant tower.
Moreover, $\mu_q(\BF(X_\infty)) =0$ for every prime $q$
and 
\[
\lambda_q(\BF(X_\infty)) = \begin{cases}
    (p-1)(q^{\val_q(p-1)}-1) & \text{ if } q\mid p-1\\
    0 & \text{ if } q\nmid p-1.
\end{cases}
\]
\end{proposition}

\subsection{Abelian groups}
\label{sec: abelian}
Let $G$ be a finite abelian group with generating set $S$.
We consider the associated directed Cayley graph $X = X(G,S)$.
In this section, we only focus on the Defect Conjecture. 
The calculation of the Iwasawa invariants in full generality is significantly more tedious than in special cases we have worked out previously.
These calculations are being carried out in a separate project by the fourth named author.
In a different setting, the Iwasawa invariants of undirected Cayley graphs of abelian groups have been studied in \cite{ghosh2025iwasawa}.

\begin{theorem}
\label{thm: defect conj for abelian}
Let $G$ be a finite abelian group.
Let $X = X(G,S)$ where $S$ is the generating set of $G$.
In the constant $\Z_p$-tower of $X$, the defect $\delta(X_m)=0$ for all $m\geq 0$.
\end{theorem}

\begin{proof}
Since $G$ is a finite abelian group, every irreducible representation of $G$ is 1-dimensional.
In particular, every eigenvalue of the adjacency matrix $A=A_X$ occurs with the same algebraic and geometric multiplicity.
If $\lambda$ is an eigenvalue of $A$ and a root of unity then the same is true for $\lambda^{-1}$.
Indeed, this is because the characteristic polynomial $\Char_A(x)$ is a polynomial with real coefficients and the algebraic multiplicities of $\lambda$ and $\lambda^{-1}$ are the same.
Set $m(\lambda)$ to be the multiplicity of $\lambda$ as a root of $\det(\Id-\theta A)$.
It follows that 
\[
m(\lambda)=a_{\lambda^{-1}}(X)=b_{\lambda^{-1}}(X)= \dim(\ker(\Id-\lambda^{-1}A)),
\]
which implies $\delta(X_m)=0$ for all $m$.
\end{proof}

\appendix

\section{SAGE Code}
\label{sage code}

Let $X$ be the Cayley graph of $D_{2n}$.
The following code was used to compute examples and develop conjectures regarding  $\mu_p(\mathcal{BF}(X_\infty))$, $\lambda_p(\mathcal{BF}(X_\infty))$, and $\delta(X_m)$, where $X_m$ is the $m$-th layer of the constant $\Z_p$ tower with base graph $X$.
This code outputs a table with columns: $n$, Size of $D_{2n}$, $\det(\mathcal{BF}(X))$, $\det(\Id-(T+1)A)$, constant term of $\det(\Id-(T+1)A)$, and the prime divisors of the constant term. 
There are several commented sections in the code block that compute and display additional information such as $\det(A_X)$, $\Char_A(x)$, $a(X)$, and $\det(\Id -\theta A)$.

In general, Cayley graphs depend on both the chosen generators for a group and the action of those generators.  Our code defines Cayley graphs of $D_{2n}$ on two generators, $\tau$ and $\sigma$ such that 
\[
D_{2n} =  \langle \tau, \sigma \mid \sigma^n = \tau^2=e, \tau\sigma\tau = \sigma^{-1}\rangle
\]
where generators act on the right.

\begin{verbatim}
#Choose m so that the largest Dihedral group considered is of size 2*m. Counter n
runs between 2...m and 2*n is the size of the dihedral group. 
m=30

#initialize data table.
rows = [[`n',`size of Dihedral Group', `Determinant of BF', `Determinant (1-(T+1)A)',
`Constant term', `Primes that divide constant term']]

for n in [2..m]:
    
    #Construct the adjacency matrix of the Cayley graph.
    D=DihedralGroup(n)
    G=D.gens()
    U=D.cayley_graph(generators=G)
    Adj=U.adjacency_matrix()
   
    #Define an identity matrix of correct size and find the Bowen-Franks matrix
    at the first level.
    I=(matrix.identity(2*n))
    BF=I-Adj
    
    #-----------------------#
    #Un-comment the this section to print out information on base graph: 
    characteristic polynomial, adjacency matrix, the characteristic polynomial 
    of the adjacency matrix, and the a(x) invariant.
    
    #f=Adj.charpoly()
    #ffactored = f.factor()
    #print(`The Cayley graph of the Dihedral group of size 2*', n, '\n )
    #print(`And characteristic polynomial', ffactored)
    #print (`And has adjacency matrix:\n',Adj)
    #print(`\n')
    #Find the multiplicity of (\lambda - 1) in the charcteristic polynomial,
    i.e. a(x), the algebraic multiplicity of 1 as an eigenvalue of the adjacency
    matrix.
    #ES=Adj.eigenspaces_right()
    #for i in [0..len(ES)-1]:
        #if 1 in Adj.eigenspaces_right()[i]:
            #dim=Adj.eigenspaces_right()[i][1].dimension()
            #print(`The dimension for lambda = 1 is',dim, '(this is a(x))')
            #print(`\n')
            
   #-----------------------#  
    #Un-comment the this section to print out information on I-H*Adj, the roots 
    of the characteristic polynomial ad their multiplicity
    
    #Define a polynomial ring in variable H and find the characteristic polynomial of
    I-H*Adj.  Here H denotes \theta the topological generator of Z_p. Factor 
    and print roots over CC.
    #A.<H>= PolynomialRing(ZZ)
    #M= I-H*Adj 
    #q= M.determinant()
    #qfactored= q.factor()
    #qroots = q.change_ring(CC).roots()
   
    #Print roots with multiplicity of the characteristic polynomial of I-HA
    #print(`The det(I-HA), where H is a variable is:', q, 'with roots, multiplicity:',
    qroots)
    #print(`\n')
   
    #-----------------------#
   
    #Subsitute T=H-1 and find characteristic polynomial of I-T*Adj.  Then finds the 
    prime divisors of the constant term. 

    B.<T>= PolynomialRing(ZZ)
    N= I-(T+1)*Adj #roots of this poly
    r= N.determinant()

    if r(0) == 0:
        primediv = [] 
        
    else:
        primediv = r(0).prime_divisors()
    
    #Note: Constant term should always equal the determinant of the Bowen-Franks group
    of the base graph.
    rows.append([n,2*n,BF.determinant(),r,r(0),primediv])
    
table(rows)
\end{verbatim}

\bibliographystyle{amsalpha}
\bibliography{references}

\end{document}